\documentclass[12pt]{amsart}
 \usepackage{mathtools}
\mathtoolsset{showonlyrefs,showmanualtags}

\usepackage[backrefs]{amsrefs}

 \usepackage[top=1.5in, bottom=1.5in, left=1.25in, right=1.25in]	{geometry}
\usepackage[all]{xy}
\usepackage{amsmath}
\usepackage{amssymb}
\usepackage{amsthm}
\usepackage{amscd}

\newcommand{\R}{\mathcal{R}}

\newcommand{\F}{\mathcal{F}}

\newcommand{\V}{\mathcal{V}}

\newcommand{\Z}{\mathbb  Z}

\newcommand{\RR}{\mathbb  R}
\newcommand{\TT}{\mathbb  T}
\newcommand{\NN}{\mathbb  N}

\newcommand{\Ca}{\mathcal{C}}

\numberwithin{equation}{section}

\newtheorem{theorem}[equation]{Theorem}
\newtheorem{definition}[equation]{Definition}
\newtheorem{proposition}[equation]{Proposition}
\newtheorem{cor}[equation]{Corollary}
\newtheorem{lemma}[equation]{Lemma}

\newtheorem{remark}[equation]{Remark}

\DeclareMathOperator{\supp}{supp}

\usepackage{hyperref} %,hypertexnames=false,colorlinks,[pagebackref]
\hypersetup{
    colorlinks=true,       % false: boxed links; true: colored links
    linkcolor=blue,          % color of internal links
    citecolor=magenta,        % color of links to bibliography
    filecolor=magenta,      % color of file links
    urlcolor=cyan           % color of external links
}

\begin{document}
\title[Discrete Monomial Carleson]{Discrete Analogoues in Harmonic Analysis: Maximally Monomially Modulated Singular Integrals Related to Carleson's Theorem}
%Enter your title between curly braces
\author{Ben Krause}
\address{
Department of Mathematics,
Caltech \\
Pasadena, CA 91125}
\email{benkrause2323@gmail.com}
\date{\today}

\date{\today}

\begin{abstract}
Motivated by Bourgain's work on pointwise ergodic theorems, and the work of Stein and Stein-Wainger on maximally modulated singular integrals without linear terms, we prove that the maximally monomially modulated discrete Hilbert transform,
\[ \mathcal{C}_df(x) := \sup_\lambda \left| \sum_{m \neq 0} f(x-m) \frac{e^{2\pi i \lambda m^d}}{m} \right| \]
is bounded on all $\ell^p, \ 2 - \frac{1}{d^2 + 1} < p < \infty$, for any $d \geq 2$. We also establish almost everywhere pointwise convergence of the modulated ergodic Hilbert transforms (as $\lambda \to 0$)
\[ \sum_{m \neq 0} T^m f(x) \cdot \frac{e^{2\pi i \lambda m^d}}{m} \]
for any measure-preserving system $(X,\mu,T)$, and any $f \in L^p(X), \ 2 - \frac{1}{d^2 +1} < p < \infty$.
\end{abstract}

\maketitle

 \setcounter{tocdepth}{1}
\tableofcontents

\section{Introduction}

Discrete analogues of (continuous) polynomial radon transforms has been an active area of research since Bourgain initiated their study in the course of his work on pointwise ergodic theorems in the late 80s and early 90s, \cite{B1,B2,B3}. In December 2015, this study was dramatically advanced in two papers by Mirek, Stein, and Trojan \cite{MST1,MST2}, where full norm estimates were proven for both maximal radon transforms, and their larger, \emph{variational}, variants; recently, this line of inquiry has been essentially concluded in the work of Mirek, Stein, and Zorin-Kranich \cite{MSZK}.

In this paper, we investigate discrete analogues of \emph{maximally modulated oscillatory singular integrals} of the type considered by Stein \cite{S} and Stein-Wainger \cite{SW}; this paper will be concerned with monomial generalizations of Stein's purely quadratic ``Carleson'' operator, known to be bounded on all $L^p, \ 1 < p < \infty$, \cite{S}:\footnote{Throughout this paper, we will refer to our maximally monomially modulated Hilbert transforms as Carleson operators. Although none of our operators are modulation invariant -- we ask the reader to forgive this abuse of notation in the interest of increased readability. Concerning modulation invariant operators, a transference argument from the continuous setting yields the boundedness of the discrete Carleson operator, $\sup_\lambda \left| \sum_{m \neq 0} f(x-m) \frac{e^{2 \pi i \lambda m}} {m} \right|,$ on all $\ell^p, \ 1 < p < \infty$ (this observation is due to Stein, \cite{SS}). It is only when \emph{linearity} is destroyed that the continuous theory and discrete theory diverge (cf.\ e.g.\ \cite{DTT}).} 
\begin{equation}\label{S}
\sup_\lambda \left| \int f(x-t) \frac{e(\lambda t^2)}{t} \ dt \right|, \ e(t) := e^{2\pi i t}.
\end{equation}

In particular, the central object of focus will be the following discrete operators,
\begin{equation} \label{C}
\Ca_d f(x) := \sup_\lambda \left| \sum_{m \neq 0} f(x-m) \frac{e(\lambda m^d)}{m} \right| = \sup_\lambda \left| \sum_{m \neq 0} f(x-m) \frac{e(-\lambda m^d)}{m} \right|.
\end{equation}

To appreciate the delicacy of these operators, note that $\ell^2$ bounds for the (two-variable) discrete Hilbert transform along variable ``parabolas''
\[ \sum_{m \neq 0} \frac{ f(x-m,y-v(x)m^d) }{m} ,\]
follows directly from corresponding $\ell^2$ estimates for $\Ca_d$ -- for any function $v: \Z \to \Z$.\footnote{This can be seen by taking a partial Fourier transform in the $y$-variable.}

Perhaps unsurprisingly, then, $\Ca_d$ has so far proven rather resistant to the ($\ell^2$) arguments developed by Bourgain and others; our main result, which builds upon a strategy devised in previous work with Michael Lacey in which the supremum in $\Ca_2$ was highly constrained \cite{KL}, establishes $\ell^2$ estimates for $\Ca_d, \ d \geq 2$. In fact, we are able to develop a full $\ell^p$ theory for $p \geq 2$, and a partial one for $1 < p \leq 2$.

\begin{theorem}\label{MAIN}
For any $d \geq 2$, and any $2 - \frac{1}{d^2 + 1} < p < \infty$, there exists an absolute constant $C_{d,p}$ so that the following norm bound holds:
\begin{equation}\label{result}
\| \Ca_d f\|_{\ell^p} \leq C_{d,p} \| f\|_{\ell^p}.
\end{equation}
\end{theorem}

A heuristic, beautifully distilled in \cite{MST1}, is that -- in light of the multiplier arguments developed first in Ionescu-Wainger \cite{IW} (see \S \ref{s:IW}) -- harmonic analysis in the discrete setting should parallel the continuous setting ``up to logarithms.'' This principle is implicit in much of the work of this paper, but becomes particularly explicit in the number-theoretic $TT^*$ argument of \S \ref{s:TT*1}, and in the \emph{multi-frequency} analysis of \S \ref{s:max}, where smoothing estimates for certain oscillatory multipliers are needed. 
However, to apply the Ionescu-Wainger theory in our multi-frequency setting, we are forced to pass to certain (continuous) square functions (introduced and discussed in \S \ref{s:SFE}), which enjoy smoothing when $p \geq 2$, but are uncontrolled in the low-$L^p$ setting; estimating $\Ca_d$ below $\ell^2$ requires an additional argument, drawing upon $\ell^2$ methods, interpolation, and M\"{o}bius inversion, see \S \ref{s:comp1}.

As a corollary of our methods, we are able to handle the issue of pointwise convergence of the associated ergodic theoretic operators. Specifically, we have the following result.
\begin{theorem}\label{pointwise}
Let $(X,\mu,T)$ be a $\sigma$-finite measure space, equipped with an invertible measure-preserving transformation, $T$. Then for any $2 - \frac{1}{d^2 + 1} < p  <\infty$, $d \geq 2$ and any $f \in L^p(X)$, there exists an $f_d \in L^p(X)$ so that
\[ \lim_{\lambda \to 0} \ \sum_{m \neq 0} T^m f \cdot \frac{e(\lambda m^d)}{m} = f_d \]
$\mu$-a.e.
\end{theorem}
We prove this result in the final section of the paper by way of an \emph{oscillation inequality}, an approach pioneered by Bourgain in his proof of pointwise convergence of ergodic averages along monomial orbits \cite{B1}. This argument is entirely $\ell^2$-based, and so is of a simpler nature.

\medskip 

The structure of the paper is as follows:

\S \ref{s:sketch} contains a top-down sketch of the argument;

\S \ref{s:ent} is concerned with preliminary tools;

In \S \ref{s:SFE} we prove square function estimates needed for our multi-frequency theory;

In \S \ref{s:max} we develop our multi-frequency theory;

In \S \ref{s:TT*1} we restrict our modulation parameters via a $TT^*$ argument;

In \S \ref{s:TT*2} we use another $TT^*$ argument to control certain ``arithmetic'' maximal functions weighted by Weyl sums;

In \S \ref{s:app} we perform our number theoretic approximations; this material is motivated by the analogous section of \cite{B1}; 

In \S \ref{s:comp}, we complete the $2 \leq p < \infty$ case of Theorem \ref{MAIN};

In \S \ref{s:comp1}, we extend the estimate to the full range, $2 - \frac{1}{d^2 + 1} < p < \infty$, thereby completing the proof of Theorem \ref{MAIN};

Finally, in \S \ref{s:pointwise}, we prove Theorem \ref{pointwise}.

We also include an appendix, \S \ref{appendix}, containing a stationary phase estimate used in \S \ref{s:SFE}.

\subsection{Acknowledgements}
First, the author wishes to thank Lillian Pierce for introducing him to $\mathcal{C}_2$. He also wishes to thank Michael Christ, Xiaochun Li, Victor Lie, Camil Muscalu, Lillian Pierce, and Jill Pipher for early discussions which helped inspire the $TT^*$ argument used in $\S \ref{s:TT*1}$ below. Additional thanks goes to Alex Iosevich, Mariusz Mirek, Terence Tao, and especially to Victor Lie, for their encouragement. Finally, a special acknowledgement is due to Michael Lacey; this paper would not have been written without his continued support, and generous exchange of ideas.

\section{Proof Overview} \label{s:sketch}
As in \cite{KL} we view $\Ca_d$ as a maximal multiplier operator, where the multipliers are given by
\begin{equation}\label{mult}
M^d(\lambda,\beta) = M(\lambda,\beta) := \sum_{m \neq 0} \frac{e(-\lambda m^d - \beta m)}{m},
\end{equation}
where $\lambda$ is the modulation parameter and $\beta$ is the frequency variable.

Following the approach of Bourgain, we use the circle method of Hardy and Littlewood to accurately approximate these multipliers on the so-called \emph{major boxes}; these boxes have a two-variable structure, and are determined by shared Diophantine properties of both $\lambda$ and $\beta$. In particular -- for each $\lambda$ the circle method produces a different approximating multiplier. To overcome this difficulty, we use the Kolmogorov-Seliverstov method of $TT^*$ to force special -- and restrictive -- arithmetic structure to the set of modulation parameters:

Roughly speaking, if we let 
\begin{equation}\label{multj} 
M_j(\lambda,\beta) := \sum_m \psi_j(m) e( - \lambda m^d - \beta m)
\end{equation}
for an appropriate smooth odd bump function 
\[ \psi_j(x) ``=" \frac{1}{x} \cdot \mathbf{1}_{|x| \approx 2^j},\]
see the subsection on notation below, then a $TT^*$ argument shows that
\begin{equation}\label{goodmultj}
\sup_{\lambda \notin X_j} \left| \left( M_j(\lambda,\beta) \hat{f}(\beta) \right)^{\vee} \right|
\end{equation}
has $\ell^p$ norm bounded by a constant multiple of $j^{-2}$ for 
\[ X_j := \left\{ \frac{a}{q} \text{ reduced} : q \leq C j^C \right\} + \{ |\lambda| \leq C j^C 2^{-dj} \}, \ C=C_{d,p} \text{ sufficiently large}\]
where $+$ denotes Minkowski sum. To appreciate the strength of this argument, note that if we set
\[ E_\epsilon := \bigcup_{j > C \log \frac{1}{\epsilon}} X_j, \]
and trivially estimate
\[ \sup_{\lambda \notin E_\epsilon} \left| \left( M_j(\lambda,\beta) \hat{f}(\beta) \right)^{\vee} \right| \leq CM_{HL}f,\]
for $j \leq C \log \frac{1}{\epsilon}$, and $M_{HL}$ the discrete Hardy-Littlewood maximal function, then we have already achieved the following significant strengthening of \cite{KL}:
\begin{lemma}\label{cheapresult}
For any $\epsilon > 0$, there exists a set $E_\epsilon = E_{\epsilon,d,p}$ with $|E_\epsilon| < \epsilon$ so that
\[ \sup_{\lambda \notin E_\epsilon} \left| \sum_{m\neq 0} f(x-m) \frac{ e(\lambda m^d) }{m} \right| \]
has $\ell^p$ norm bounded by a constant multiple (depending on $d,p$) of $\log \frac{1}{\epsilon}$.
\end{lemma}

More important for our purposes than this relatively cheap lemma is that the small measure of each set $X_j$ allows one to effectively apply a Sobolev-embedding argument and pass from $M(\lambda,\beta)$ to an analytic approximate. Essentially -- we have forced the major boxes at each scale to live near only a ``few" $\lambda$-frequencies.

On each major box $M(\lambda,\beta)$ now looks like a shifted version of the continuous multiplier
\[ \int e( - \lambda t^d - \beta t) \ \frac{dt}{t} \]
weighted according to the Diophantine properties of the centers of each major box. But, as many boxes arise, this multiplier has many different distinguished frequency points. This phenomenon was first encountered by Bourgain in his work on pointwise ergodic theorems, which lead him to prove estimates for certain \emph{multi-frequency} maximal averaging operators \cite[\S 4]{B3}. In \cite{KL}, an analogous oscillatory multi-frequency operator was introduced. By combining ideas from Stein-Wainger \cite{SW}, and using estimates for Bourgain's maximal function, estimating the oscillatory multi-frequency operator was reduced to a (single-frequency) maximal multiplier theorem \cite[Lemma 2.9]{KL} in a certain ``critical'' range of parameters, determined by level sets of certain phases relative to the number of distinguished frequencies.

In our setting, analogous multi-frequency operators arise, which we are now forced to handle on $\ell^p$ as well. The ideas of Stein-Wainger neatly extend, and by using the \emph{rationality} of our (carefully chosen) set of $\beta$-distinguished frequencies, we may use the techniques of \cite{MST1} to reduce the problem to understanding our operator in an analogous critical range of parameters, see \eqref{e:2} -- the \emph{``stationary'' critical range} -- and \eqref{e:3} -- the \emph{``oscillatory'' critical range} -- below. We are able to use the maximal functions of Bourgain to turn the ``stationary'' operators into (essentially) \emph{vector-valued multi-frequency multipliers}, which we can estimate by using the transference arguments of Mirek, Stein, and Zorin-Kranich \cite{MSZK}, see Theorem \ref{Multtheorem} below, and the vector-valued Mikhlin multiplier theorem. Much of this approach transfers to the ``oscillatory'' operators. But, it is here that the oscillatory vs.\ radon nature of the problem makes itself felt, as the singular integral techniques used to handle the ``stationary'' operators do not apply when certain phases have critical points. We are able to handle these terms in $\ell^p, \ p \geq 2$, by appealing to certain (single-frequency) square function estimates of \cite{LRS}, which we transfer to the multi-frequency setting upon another application of the transference argument of \cite{MSZK}. Unfortunately, this range of $p$ is sharp for this approach (see \S \ref{s:SFE} below). Indeed, to push our estimates below $p = 2$ we will need to \emph{re-select} our distinguished $\beta$-frequencies to form unions of acceptably many \emph{cyclic subgroups}, $\mathbb{Z}/Q \mathbb{Z} \subset \mathbb{T}$, see \S \ref{s:comp1}; the pertaining multi-frequency operators are controlled by their single-frequency counterparts, and are therefore estimated on each $\ell^p, \ 1<p < \infty$,
which allows us to interpolate below $p=2$.\footnote{Strictly speaking, our frequencies form unions of \emph{reduced} elements of $\mathbb{Z}/Q \mathbb{Z}$, but this situation is easily reduced to the cyclic setting by an application of M\"{o}bius inversion.}

With the multi-frequency estimates in hand, the next obstacle is that, as the rational approximations to $\lambda$ change, the distinguished $\beta$-frequencies change as well. By crucially exploiting orthogonality properties of certain Weyl sums, we may lift this restriction at the expense of another $TT^*$ argument. With these obstructions dealt with, we are able to complete the proof of Theorem \ref{MAIN}; Theorem \ref{pointwise} follows from our $\ell^2$ theory and a variational estimate of \cite{GRY}.

\section{Preliminaries}\label{s:ent}
\subsection{Notation}
First, for ease of presentation we will choose to work with the definition of $\Ca_d$ involving a negative in the exponential:
\[ \Ca_d f(x) := \sup_\lambda \left| \sum_{m \neq 0} f(x-m) \frac{e(-\lambda m^d)}{m} \right|.\]

Here and throughout, $e(t) := e^{2\pi i t}$; $x \equiv y$ will denote equivalence $\mod 1$. Throughout, $C$ will be a large number which may change from line to line. Since $\Z$ is countable, there is no loss of generality in restricting our set of modulation parameters to a countable set; this will allow us to dispose of all measurability issues. $M_{HL}$ will denote the Hardy-Littlewood maximal function on $\mathbb{Z}$ or $\mathbb{R}$ (context will distinguish which).

For finitely supported functions on the integers, we define the Fourier transform
\[ \F_{\mathbb{Z}} f(\beta) := \hat{f}(\beta) := \sum_n f(n) e(-\beta n),\]
with inverse
\[ \F_{\mathbb{Z}}^{-1}g(n) := g^{\vee}(n) := \int_{\TT} g(\beta) e(\beta n) \ d\beta.\]
For Schwartz functions on the line, we define the Fourier transform
\[ \F_{\mathbb{R}} f(\xi) := \hat{f}(\xi) := \int f(x) e(-\xi x) \ dx,\]
with inverse
\[ \F_{\mathbb{R}}^{-1}g(x) := g^{\vee}(x) := \int g(\xi) e(\xi x) \ d\xi.\]
Occasionally, for functions of two variables, $f(x,y)$, we will let
\[ \left(\F_x f(\cdot,y)\right) (\xi) := \int  f(x,y) e(-\xi x) \ dx,\]
and similarly for $\F_x^{-1}, \F_y, \ \F_y^{-1}$.

We decompose 
\[ \frac{1}{x} \cdot \mathbf{1}_{|x| \geq 1} = \sum_{j \geq 1} 2^{-j} \psi(2^{-j} x) \cdot \mathbf{1}_{|x| \geq 1} \]
for an appropriate, compactly supported, odd bump function, $\psi$.
We will set
\[ \psi_j(x) := 2^{-j} \psi(2^{-j} x).\]

We will let $\Theta$ be an even non-negative compactly supported Schwartz function, adapted to an annulus away from the origin, so that
\begin{equation}\label{Theta}
\mathbf{1}_{\xi \neq 0} = \sum_{j} \Theta_j(\xi) := \sum_j \Theta(2^j \xi),
\end{equation}

We will let $\chi$ be an even non-negative compactly supported Schwartz functions which satisfies
\begin{equation}\label{chi}
\mathbf{1}_{|\xi| \leq c_d} \leq \chi \leq \mathbf{1}_{|\xi| \leq 2c_d}
\end{equation}
for a sufficiently small constant $c_d$.

We will let $\zeta$ denote various smooth approximations to ``fat'' annuli:
\begin{equation}\label{zeta}
\mathbf{1}_{|\xi| \approx 1} \leq \zeta \leq \mathbf{1}_{|\xi| \approx 1}
\end{equation}
for some sufficiently large implicit constants.

We define $\overline{\Theta}$ to be like $\Theta$, but one on its support, and similarly define $\overline{\chi}$ and $\overline{\zeta}$.

Since the goal of this paper will be to prove a priori norm estimates, we will restrict every function considered to be a member of a “nice” dense subclass: each function on the integers will be assumed to have
finite support, and each function on the line will be assumed to be a Schwartz function. We will use
\[ \| f \|_{\ell^p} := \left( \sum_{x \in \mathbb{Z}} |f(x)|^p \right)^{1/p} \]
and
\[ \| f\|_p := \left( \int_{\mathbb{R}} |f(x)|^p \ dx \right)^{1/p},\]
with the obvious modifications at $p=\infty$.

We will make use of the modified Vinogradov notation. We use $X \lesssim Y$, or $Y \gtrsim X$, to denote the estimate $X \leq CY$ for an absolute constant $C$. We use $X \approx Y$ as shorthand for $Y \lesssim X \lesssim Y$. We also make use of big-O notation: we let $O(Y )$ denote a quantity that is $\lesssim Y$. If we need $C$ to depend on a parameter, we shall indicate this by subscripts, thus for instance $X \lesssim_p Y$ denotes the estimate $X \leq C_p Y$
for some $C_p$ depending on $p$. We analogously define $O_p(Y)$.

\subsection{Transference}
We will need the following special case of a beautiful transference argument of Magyar, Stein, and Wainger \cite[Lemma 2.1]{MSW}.

\begin{lemma}\label{trans}
Let $B_1,B_2$ be finite-dimensional Banach spaces, and 
\[ m: \RR \to L(B_1,B_2) \]
be a bounded function supported on a cube with side length one containing the origin that acts as a Fourier multiplier from
\[ L^p(\RR,B_1) \to L^p(\RR, B_2),\]
for some $1 \leq p \leq \infty.$ Here, $L^p(\RR,B):= \{ f: \RR \to B : \| \| f\|_B \|_{L^p(\RR)} < \infty\}$.
Define
\[ m_{\text{per}}(\beta) := \sum_{l \in \Z} m(\beta - l) \ \text{ for } \beta \in \TT.\]
Then the multiplier operator
\[ \| m_{\text{per}} \|_{\ell^p(\Z,B_1) \to \ell^p(\Z,B_2)} \lesssim \| m \|_{L^p(\RR,B_1) \to L^p(\RR,B_2)}.\]
The implied constant is independent of $p, B_1,$ and $B_2$.
\end{lemma}

We will use this lemma in \S \ref{s:max} below.

We next recall the following multi-frequency multiplier theorems, which in turn grew out of \cite{IW}.

\subsection{A Multi-Frequency Multiplier Theorem for Ionescu-Wainger Type Multipliers}\label{s:IW}
The results of this section appear as the special one-dimensional case of \cite[Theorem 5.7]{MSZK}, the Hilbert-space extension of \cite[Theorem 5.1]{MST1}.
\begin{theorem}[Special Case]\label{Multtheorem}
Suppose $H$ is a Hilbert space, and suppose that $m(\xi)$ is an $H$-valued (bounded) $L^p(\RR)$ multiplier with norm $A$:
\[ \| \ \left| \left( m(\xi) f(\xi) \right)^{\vee} \right|_H \ \|_{L^p(\RR)} \leq A \| f \|_{L^p(\RR)}.\]
Let $\rho >0$ be arbitrary (for later applications, we will take $0<\rho \ll_{p,d} 1$). Then, for every $N$, there exists an absolute constant $C_\rho > 0$ so that one may find a set of rational frequencies
\[ \left \{ \frac{a}{q} \text{ reduced} : q \leq N \right \} \subset \mathcal{U}_N \subset 
\left \{ \frac{a}{q} \text{ reduced} : q \leq C_\rho e^{N^{\rho}} \right \},\]
so that
\begin{equation}\label{IWM}
\left| \left( \sum_{\theta \in \mathcal{U_N}} m(\beta - \theta) \eta_N(\beta - \theta) \hat{f}(\beta) \right)^{\vee} \right|_H
\end{equation}
has $\ell^p$ norm $\lesssim_{\rho,p} A \cdot \log N$. Here, $\eta_N$ is a smooth bump function supported in a ball centered at the origin of radius $\leq e^{-N^{2\rho}}$.
\end{theorem}
\begin{remark}
Although the results of Ionescu and Wainger \cite[Theorem 1.5]{IW} yield the analogous result with an operator norm of
\[ \| \eqref{IWM} \|_{\ell^p} \lesssim_{p,\rho} A \cdot \log^{2/\rho} N \|f\|_{\ell^p}, \]
their result is only for the scalar case, which would be insufficient for our purposes.

Both results should be contrasted with the strongest analogous multiplier theorem for \emph{general} frequencies, which accrues a norm loss of
\[ \left( \text{Number of Frequencies} \right)^{|1/2-1/p|},\]
even in the special case when $m \in \mathcal{V}^2(\mathbb{R})$ has finite \emph{$2$-variation}, see \cite[Lemma 2.1]{C+}.
\end{remark}

We will frequently use this theorem in conjunction with the following vector-valued version of the Mikhlin multiplier theorem, which, roughly speaking, asserts that (bounded) multipliers which are essentially constant on dyadic annuli are bounded on $L^p, \ 1 < p < \infty$; the proof of the scalar case extends directly to the vector-valued setting.
\begin{proposition}[Vector-Valued Mikhlin Multiplier Theorem, Special Case]\label{VVMT}
Suppose that $H$ is a Hilbert space, and that $m(\xi)$ is an $H$-valued multiplier with
\begin{equation}\label{MIK}
|m(\xi)|_H + |\xi| | \partial_\xi m(\xi) |_H \leq A.
\end{equation}
Then for any $1 < p < \infty$,
\[ \| \ \left| \left( m(\xi) \hat{f}(\xi) \right)^{\vee} \right|_H \ \|_p \lesssim_p A \|f\|_p.\]
\end{proposition}
We will refer to the best constant, $A$, in \eqref{MIK}, as the \emph{Mikhlin multiplier norm} of $m$.

We next turn to more analytic considerations.

\subsection{A Sobolev Embedding Calculation}
Suppose
\begin{equation}\label{sobsit}
X_j = \bigcup_{i=1}^{Cj^C} I_i \subset [0,1]
\end{equation}
where each interval $I_i$ has length $|I_i| \lesssim j^C 2^{-dj}$. 
Suppose further that $F$ is a $\mathcal{C}^1$ function from $[0,1] \times [0,1] \to \mathbb{C}$ with
\begin{equation}\label{mult}
\sup_\lambda \|\left( F(\lambda, \beta) \hat{f}(\beta) \right)^{\vee}\|_{\ell^p} \leq a(p) \| f \|_{\ell^p}
\end{equation}
and
\begin{equation}\label{mult}
\sup_\lambda \|\left( \partial_\lambda F(\lambda, \beta) \hat{f}(\beta) \right)^{\vee}\|_{\ell^p} \leq A(p) \| f \|_{\ell^p},
\end{equation}
where $\partial_\lambda$ denotes the partial derivative with respect to the $\lambda$ variable, and the second supremum is taken only over $\lambda$ that are in the interior of $X_j$.

Then we have the following Sobolev-embedding type lemma.
\begin{lemma}\label{sobemb}
Under the above conditions, for any $1 \leq p \leq \infty$,
\[ \| \sup_\lambda |\left( \partial_\lambda F(\lambda, \beta) \hat{f}(\beta) \right)^{\vee}| \|_{\ell^p} \lesssim_p \left( j^C a(p) + j^C 2^{-dj/p} a(p)^{1-1/p} A(p)^{1/p} \right) \cdot \| f \|_{\ell^p}.\]
\end{lemma}
\begin{proof}
The $\ell^\infty$ estimate is trivial, so we assume $1 \leq p < \infty$.

Since we are free to lose factors of $j^C$, we may use the triangle inequality to restrict to a single interval $|I| \lesssim j^C 2^{-dj}$; since we are free to lose factors of $a(p)$, we may estimate the contribution of each endpoint of $I$ independently (if $I$ is (half) closed). Consequently, we will henceforth assume that $\lambda \in I$ is in the \emph{interior}.

Now, with $\lambda_I$ the left end-point of $I$, we write
\[ \aligned 
\left( \left( F(\lambda, \beta) \hat{f}(\beta) \right)^{\vee} \right) ^p &= 
\left( \left( F(\lambda_I, \beta) \hat{f}(\beta) \right)^{\vee}\right)^p \\
& \qquad +
p \int_{[\lambda_I,\lambda]} \left( \left( F(t, \beta) \hat{f}(\beta) \right)^{\vee} \right)^{p-1} \cdot \left( \partial_t F(t, \beta) \hat{f}(\beta) \right)^{\vee} \ dt \endaligned \]
and apply H\"{o}lder to dominate, for each $x \in \mathbb{Z}$,
\[ \aligned 
&\sup_{\lambda \in I} \left| \left( F(\lambda, \beta) \hat{f}(\beta) \right)^{\vee}(x) \right| ^p \\
& \qquad \lesssim_p \left| \left( F(\lambda, \beta) \hat{f}(\beta) \right)^{\vee}(x) \right| ^p \\
& \qquad \qquad + \left( \int_I \left| \left( F(\lambda, \beta) \hat{f}(\beta) \right)^{\vee}(x) \right|^{p} \ d\lambda \right)^{\frac{1}{p'}} \cdot \left( \int_I \left| \left( \partial_\lambda F(\lambda, \beta) \hat{f}(\beta) \right)^{\vee}(x) \right|^{p} \ d\lambda \right)^{\frac{1}{p}}. \endaligned\]
Summing over $x \in \mathbb{Z}$ and applying H\"{o}lder once more yields the result.
\end{proof}

\section{Square Function Estimates}\label{s:SFE}
The goal of this section is to prove a (single-frequency) square function estimate, which will be used in the \emph{``oscillatory'' critical regime}, see \eqref{e:3} below, when singular integral techniques break down. In this section, we will work entirely in the high $L^p, \ p \geq 2$ regime.

First, some notation.

Throughout, $l \geq 1$ will be a positive integer, and $k$ will be another integer which satisfies the relationship
\begin{equation}\label{kvsl}
k^C \gtrsim 2^l.
\end{equation}

We will use the following abbreviation,
\begin{equation}\label{Fracpower}
\xi^{d/(d-1)} :=
\begin{cases} |\xi|^{d/(d-1)} &\mbox{if } d \text{ is odd} \\ 
 \text{sgn}(\xi)|\xi|^{d/(d-1)} & \mbox{if } d \text{ is even}. \end{cases}
\end{equation}

Now, for $2^{l-dk} \leq \lambda < 2^{l-dk+1}$, and for each $|\xi| \approx 2^{l-k}$,
we define the phase
\begin{equation}\label{phase}
\varphi^k(t,\xi,\lambda) := \varphi^k(t,\xi) := - 2^{-l} \left( \lambda 2^{kd} t^{d} + \xi 2^k t\right);
\end{equation}
we will be interested in estimating
\begin{equation}\label{invF}
G_\lambda(x):=\left( \int e(2^l \cdot \varphi^k(t,\xi)) \psi(t) \ dt \cdot \zeta(2^{k-l} \xi) \right)^{\vee}(x) = e(-\lambda \cdot^d) \psi_k(\cdot ) * \left( \zeta(2^{k-l} \cdot ) \right)^{\vee}(x),
\end{equation}
where $\zeta$ is as in \eqref{zeta}.
In particular, the goal of this section will be to estimate the following square functions:
\begin{equation}\label{SFXN}
S_G f:= \left( \sum_k 2^{dk} \int_{2^{l-dk}}^{2^{l-dk +1}} |G_\lambda *f|^2 \ d\lambda \right)^{1/2}
\end{equation}
and
\begin{equation}\label{SFXND}
S_{G'} f := \left( \sum_k 2^{-dk} \int_{2^{l-dk}}^{2^{l-dk +1}} |\partial_\lambda G_\lambda *f|^2 \ d\lambda \right)^{1/2}.
\end{equation}

\begin{theorem}\label{SINGFREQEST}
For any $2 \leq p < \infty$, one has the estimate
\[ \| S_G f \|_p + \| S_{G'} f \|_p \lesssim_p l \|f\|_p.\]
\end{theorem}
Noting that
\[ 2^{-dk} \partial_\lambda G_\lambda \]
is essentially the same object as $G_\lambda$, it will suffice only to estimate $S_G f$.

This theorem will be proven over the following sub-sections. We begin our discussion by developing some auxiliary square function estimates which, we will see, are ``morally'' equivalent to $S_G f$ (see Lemma \ref{decomp} below).

\subsection{The Main Contribution}
For $I$ a compact interval supported away from the origin of length $|I| \approx_d 1$, let $\mathcal{K}$ denote the (Hilbert-space-valued) linear operator
\[ \aligned 
{\mathcal{K}f} &:= \{ \mathcal{K}_{\lambda,k} f : \lambda \in I, k \in \mathbb{Z} \} \\
& \qquad := \left\{ 2^{\frac{l-dk}{d}} \times
 \int e( 2^{\frac{l-dk}{d}} x \xi - \lambda \xi^{d/(d-1)}) \left( \hat{f}(2^{\frac{l-dk}{d}} \cdot \xi) \zeta(2^{\frac{l-dk}{d}} \cdot 2^{k-l} \xi) \right)  \ d\xi  : \lambda \in I, k \in \mathbb{Z} \right\} \\
 & \qquad \qquad := \left\{ K_{\lambda,k}*f : \lambda \in I, k \in \mathbb{Z} \right\},
\endaligned \]
where we define
\[ K_{\lambda,k}(x) := \int e(x \xi - \lambda 2^{\frac{dk-l}{d-1}} \xi^{ \frac{d}{d-1}}) \zeta(2^{k-l} \xi) \ d\xi = 2^{l-k} \int e(2^{l-k}x \xi - \lambda 2^l \xi^{d/d-1}) \zeta(\xi) \ d\xi.\]
In particular, our Hilbert space consists of functions of the form
\[ F(x) := \{ F_{\lambda,k}(x) : \lambda \in I, \ k \in \mathbb{Z}\},\]
and we define our norm
\[ | F(x)|^2_H := \sum_k \int_I |F_{\lambda,k}(x)|^2 \ d\lambda,\]
so we have
\[ |\mathcal{K}f|_H^2 = \sum_k \int_I |\mathcal{K}_{\lambda,k} f|^2 \ d\lambda.\]
We will set
\begin{equation}\label{Kk}
\mathcal{K}_k f := \{ \mathcal{K}_{\lambda,k} f : \lambda \in I \},
\end{equation}
with norm
\[ |\mathcal{K}_k f|_I^2 :=  \int_I |\mathcal{K}_{\lambda,k} f|^2 \ d\lambda = \| \mathcal{K}_{\lambda,k} f \|_{L^2(\lambda \in I)}^2.\]
By Plancherel, we quickly deduce the following $L^2$ estimate on $\mathcal{K}$.
\begin{proposition}\label{allscalesL2}
We have the $L^2$ estimate,
\[ \| \ \left|\mathcal{K}f \right|_H \ \|_2 \lesssim \|f \|_2.\]
\end{proposition}
Below $L^2$, no such result can hold. In fact, this can be seen at the single scale level.
In particular, if we define
\begin{equation}\label{Ker}
K_{\lambda,k}^0(x) := \int e( x \xi - \lambda \xi^{d/(d-1)} ) \zeta(2^{-l \frac{d-1}{d} } \xi) \ d\xi =
2^{l \frac{d-1}{d} } \int e( 2^{l \frac{d-1}{d} } x \xi - 2^l \lambda \xi^{d/(d-1)} ) \zeta(\xi) \ d\xi,
\end{equation}
then
\[ \mathcal{K}_{\lambda,k} f(x) := D_{(dk-l)/d} \Big( K_{\lambda,k}^0 * \big( D_{(l-dk)/d} f \big) \Big),\]
where
\[ D_a g(x)  := 2^{-a} g(2^{-a} x)\]
are $L^1$-normalized dilations. Consequently, if $C_{p,k}$ is the best constant in the estimate
\begin{equation}\label{bestconst}
C_{p,k} := \sup_{\| f\|_p = 1} \| \ | \mathcal{K}_k f |_I \ \|_p  = \sup_{\| f\|_p = 1} \| \| K_{\lambda,k}^0 * f\|_{L^2(\lambda \in I)} \|_p ,
\end{equation}
then by specializing to $\hat{f}(\xi) := \overline{\zeta}(2^{-l  \frac{d-1}{d} }\xi)$, we use the principle of stationary phase when $|x| \approx 2^{l/d}$ to estimate
\[ \| 2^{l(1/2 - 1/d)} \mathbf{1}_{|x| \approx 2^{l/d}} \|_p \lesssim \| \| K_{\lambda,k}^0 \|_{L^2(\lambda \in I)} \|_p \leq C_{p,k} \| D_{-l  \frac{d-1}{d} } (\overline{\zeta})^{\vee} \|_p \]
which leads to an exponential blow-up in norm unless $p \geq 2$.\footnote{As we will see below, to develop an $\ell^p$ theory below $p=2$ using the approach which we will use to develop our $\ell^p$ theory for $p \geq 2$, we would need a blow up on the order of $l^{O_{p,d}(1)}$.}

On the other hand, it turns out that -- up to acceptable losses in $l$ -- $\mathcal{K}$ is bounded on $L^p, \ 2 \leq p < \infty$. In particular, we have the following proposition.

\begin{proposition}\label{allscaleshighLp}
For any $2 \leq p < \infty$, we may bound
\[ \| \ \left|\mathcal{K}f\right|_H \ \|_p \lesssim l \|f\|_p.\]
\end{proposition}

To establish Proposition \ref{allscaleshighLp} we will need the following lemma, which records the relevant estimates on $K_{\lambda,k}$.
\begin{lemma}\label{deriv/decay}
Set 
\[ \mathcal{T}_k(x) := \begin{cases} 
2^{l/2 - k} & \mbox{if } |x| \lesssim 2^k \\
2^{l-k} (2^{l-k} |x|)^{-N} & \mbox{if } |x| \gg 2^k. \end{cases} \]
Then
\begin{equation}
|\left( \frac{d}{dx}\right)^j K_{\lambda,k}(x)| \lesssim 2^{(l-k)j} \cdot \mathcal{T}_k(x)
\end{equation}
for each $j = 0, 1$.
\end{lemma}
\begin{proof} It is enough to prove the $j=0$ case, as the $j=1$ case is similar.
But, the second derivative of the phase has magnitude about $2^l$, which yields the estimate when $|x| \lesssim 2^k$ by the principle of stationary phase. When $|x| \gg 2^k$, the result follows from the principle of non-stationary phase.
\end{proof}

The first step in proving Proposition \ref{allscaleshighLp} will be estimating $\mathcal{K}_k f$ on $L^p, \ 2\leq p < \infty$, for which we will need the following square function estimate, due to Lee, Rogers, and Seeger \cite{LRS}.
\begin{proposition}[Proposition 5.2 of \cite{LRS}]
Let $p \geq 2$ and $\alpha > 1$. Then for any compact time interval $I$,
\begin{equation}\label{SMOOTHING0}
\| \left( \int_I \left| \int e(x \xi) \hat{f}(\xi) e(t |\xi|^\alpha) \ d\xi \right|^2 \ dt \right)^{1/2} \|_p \lesssim_{p,I} \| f \|_p.
\end{equation}
\end{proposition}
Using this proposition, we quickly deduce the following estimate concerning each individual operator $\mathcal{K}_k$ \eqref{Kk}.
\begin{lemma}\label{SINGLET_k}
For any $2 \leq p < \infty$,
\[ \| \ \left| \mathcal{K}_k f \right|_I \ \|_p \lesssim \| f \|_p.\]
\end{lemma}
\begin{proof}
By standard $L^p$ estimates for dilates of a function, see \eqref{bestconst} above, matters reduce to estimating -- at worst --
\[ \| \left( \int_I \left| \int e(x \xi) \hat{f}(\xi) e(\pm t \cdot \text{sgn}(\xi) |\xi|^{d/(d-1)}) \ d\xi \right|^2 \ dt \right)^{1/2} \|_p \lesssim_{p,I} \| f \|_p. \]
But, this follows from \eqref{SMOOTHING0} and the $L^p$ boundedness, 
\[ \| \left( \hat{f} \mathbf{1}_{\pm \xi > 0} \right)^{\vee} \|_p \lesssim \|f\|_p.\]
\end{proof}
To upgrade Lemma \ref{SINGLET_k} to Proposition \ref{allscaleshighLp}, we will use an argument of Seeger, \cite[Theorem 1]{See}, which appeared in the vector-valued setting as Proposition 4.3 of \cite{GRY}. To do so, we will need to use (Banach-space valued) sharp functions:
\[ M^{\#}f(x) := \sup_{Q \ni x \text{ dyadic}} \frac{1}{|Q|} \int_Q \Big| f(y) - [f]_Q \Big|_B \ dy,\]
where $[f]_Q := \frac{1}{|Q|} \int_Q f(y) \ dy$ is the average of $f$ over the cube $Q$. As in the Euclidean setting, one has the estimates
\[ \| M_{HL} f\|_p \approx_p \| M^{\#} f \|_p, \ 1 \leq p < \infty,\]
as the standard good-$\lambda$ argument transfers; see \cite[Lemma B.1]{GRY} for details.

In particular, it suffices now to estimate
\[ M^{\#}(\mathcal{K}f) \]
in $L^p$.
\begin{proof}[Proof of Proposition \ref{allscaleshighLp}]
Suppose for concreteness that
\[ M^{\#} (\mathcal{K}f) (x) = \sup_{P \ni x} \frac{1}{|P|} \int_P \left|\mathcal{K}f - [\mathcal{K}f]_P \right|_H \]
is realized by the particular average
\[ \frac{1}{|Q|} \int_Q \left|\mathcal{K}f - [\mathcal{K}f]_Q \right|_H, \ Q = Q(x),\]
and decompose
\[ \mathcal{K}f = \mathcal{K}^1 f + \mathcal{K}^2 f,\]
where $\mathcal{K}^1$ is the part of the operator $\mathcal{K}$ which lives at scales near $q$, where $|Q| = 2^q$, and $\mathcal{K}^2$ is the complementary component. In particular,
\[ \mathcal{K}^1 f := \{ \mathcal{K}_{\lambda,k} f : \lambda \in I,  |k-q| \lesssim l \}\]
for some sufficiently large (absolute) implicit constant; the $\mathcal{K}^2$ contribution is an error term.

We trivially estimate $\|\mathcal{K}^1f\|_p \lesssim l \|f\|_p$ by dominating the sharp function by a constant multiple of Hardy Littlewood:
\[ \aligned 
\frac{1}{|Q|} \int_Q \left|\mathcal{K}^1f - [\mathcal{K}^1f]_Q \right|_H &\leq \sum_{k:|k-q| \lesssim l} \frac{1}{|Q|} \int_Q \|\mathcal{K}_{\lambda,k} f - [\mathcal{K}_{\lambda,k}f]_Q\|_{L^2(\lambda \in I)} \\
& \qquad \lesssim \sum_{|k-q| \lesssim l} M_{HL} |\mathcal{K}_k f|_I, \endaligned \]
from which the result follows.

We now estimate $\mathcal{K}^2f$ on $L^2$ and on $L^\infty$. The $L^2$ estimate is straightforward, as we may dominate
\[ \frac{1}{|Q|} \int_Q \left|\mathcal{K}^2f - [\mathcal{K}^2 f]_Q \right|_H \lesssim M_{HL} \left|\mathcal{K}f\right|_H.\]
To derive the $L^\infty$ estimate, we split 
\[ f = f_0 + f_{\infty},\] 
where $f_0 := f \cdot \mathbf{1}_{CQ}$ for some sufficiently large constant $C$; here $CQ$ is the $C$-fold dilate of $Q$ about its center. Then we may estimate
\[ \aligned 
\frac{1}{|Q|} \int_Q \left|\mathcal{K}^2f_0 - [\mathcal{K}^2f_0]_Q \right|_H &\lesssim \left( \frac{1}{|Q|} \int_Q \left|\mathcal{K}^2f_0 - [\mathcal{K}^2f_0]_Q \right|_H^2 \right)^{1/2} \\
& \qquad \lesssim |Q|^{-1/2} \| f_0 \|_2 \\
& \qquad \qquad \lesssim \|f\|_\infty.
\endaligned \]
So, we need to bound 
\[ \frac{1}{|Q|} \int_Q \left|\mathcal{K}^2f_{\infty} - [\mathcal{K}^2 f_{\infty}]_Q \right|_H \]
from above. In particular, it suffices to simply estimate
\begin{equation}\label{Kernelest}
\sum_{k: |k-q| \gg l} \sup_{y,z \in Q} \int_{w \notin CQ} \sup_{\lambda \in I} |K_{\lambda,k}(y-w) - K_{\lambda,k}(z-w)| \ dw.
\end{equation}
In the case where $k \leq q - Cl$, we may bound
\[ \aligned
\eqref{Kernelest} &\lesssim \sum_{k \leq q - Cl} \int_{|w| \gg 2^q} \mathcal{T}_k(w) \\
& \qquad \lesssim \sum_{k \leq q - Cl} 2^{N(k-l)} 2^{-qN} \\
& \qquad \qquad \lesssim 2^{-lN}, \endaligned \]
where we used that $\mathcal{T}_k \cdot \mathbf{1}_{|x| \lesssim 2^k}$ vanishes identically on the domain of integration.
In the case where $k \geq q + Cl$, we may bound
\[ \aligned
\eqref{Kernelest} &\lesssim \sum_{k \geq q + Cl} 2^{q + l-k} \int \mathcal{T}_k(w) \\
& \qquad \lesssim \sum_{k \geq q + Cl} 2^{q + 3l/2-k}  \\
& \qquad \qquad \lesssim 2^{-lC}, \endaligned \]
which completes the proof (note how we used the radially-decreasing nature of $\mathcal{T}_k$).
\end{proof}

With these estimates in hand we are almost ready to prove Theorem \ref{SINGFREQEST}.
\subsection{The Proof of Theorem \ref{SINGFREQEST}}
Let us assume the following lemma, whose proof will be deferred to \S \ref{appendix} below.

\begin{lemma}\label{decomp}
For any (large) $N$, one may decompose $G_\lambda = A_\lambda + \sum'_{\pm} B^{\pm}_{\lambda}$, which satisfy the following estimates, independent of $\lambda$:
\[ |A_\lambda*f| \lesssim_N 2^{-lN} M_{HL} \left( \overline{\zeta}(2^{k-l}\cdot) \hat{f} \right)^{\vee} \]
pointwise,
and
\[ \widehat{B^{\pm}_{\lambda}}(\xi) = 2^{-l/2} \cdot e( \pm c_d \lambda^{-1/(d-1)} \xi^{d/(d-1)} ) \cdot m(\xi,\lambda) \cdot \zeta(2^{k-l}\xi),\]
for some $|c_d| \approx_d 1$; in the case where $d$ is odd, we replace $\zeta$ with $\zeta \cdot \mathbf{1}_{\xi < 0}$ throughout (which satisfies all the same differential inequalities as does $\zeta$ itself).
Here
\[ \sup_\lambda |\partial_\xi^j m(\xi,\lambda)| \lesssim_j |\xi|^{-j}, \ j \geq 0, \ \xi \neq 0.\]
In particular, we may decompose
\[ \widehat{B^{\pm}_{\lambda}}(\xi) = 
\widehat{O^{\pm}_\lambda(\xi)} \widehat{ M_\lambda (\xi) },\]
where
\[ \widehat{O^{\pm}_\lambda(\xi)} := 2^{-l/2} \cdot e( \pm c_d \lambda^{-1/(d-1)} \xi^{d/(d-1)} ) \cdot \zeta(2^{k-l}\xi),\]
and
\[ \widehat{ M_\lambda (\xi) } := m(\xi,\lambda) \cdot \overline{\zeta}(2^{k-l}\xi) \]
satisfies
\[ |M_\lambda(x) | \lesssim_N 2^{l-k} (1 + |2^{l-k} x|)^{-N}.\]
Here $\sum'_{\pm}$ means that the ``minus'' term appears only when $d$ is odd.
\end{lemma}
%\begin{remark} Since $t^d \psi(t)$ satisfies the same differential inequalities as does $\psi(t)$ itself, we see that we have an analogous decomposition for $2^{-dk} \partial_\lambda G_\lambda$.
%\end{remark}

With this decomposition lemma in hand, we are able to quickly complete the proof.
\begin{proof}[Proof of Theorem \ref{SINGFREQEST}]
If we decompose $G_\lambda = A_\lambda + \sum_{\pm}' B_\lambda^{\pm}$, and estimate
\[ S_G f \leq S_A f + \sum_{\pm}' S_{B^{\pm} f},\]
where
\[ \aligned 
|S_A f|^2 &:= \sum_k 2^{dk} \int_{2^{l-dk}}^{2^{l-dk+1}} | A_\lambda*f|^2 \ d\lambda \ \text{ and} \\
|S_{B^{\pm}} f|^2 &:= \sum_k 2^{dk} \int_{2^{l-dk}}^{2^{l-dk+1}} | B_\lambda^{\pm}*f|^2 \ d\lambda , \endaligned\]
then we may estimate
\[ \|S_A f\|_p \lesssim 2^{-lN} \|f\|_p,\]
by the Fefferman-Stein inequalities and the boundedness of the Littlewood-Paley square function. We only treat $B_\lambda^+$, as $B_\lambda^{-}$ -- if it's present -- is handled similarly. By the Fefferman-Stein inequalities, we may replace $S_{B^+} f$ by 
\[ |S_{O^{+}} f|^2 := \sum_k 2^{dk} \int_{2^{l-dk}}^{2^{l-dk+1}} | O_\lambda^{+}*f|^2 \ d\lambda .\]
But now the result follows from Propostion \ref{allscaleshighLp} by a change of variables; the key point is that the map $t^{-1/(d-1)} \mapsto t$ has a (harmless) bounded Jacobian on $t \approx_d 1$. 
\end{proof}

\section{A Key Maximal Inequality} \label{s:max}
We present and prove a key maximal inequality used in the proof of Theorem \ref{MAIN}, an extension of the maximal inequality of \cite[\S 3]{KL}, which in turn is an extension of Bourgain \cite[\S 4]{B3}, 
the harmonic analytic core of the proof of the arithmetic ergodic theorems. Before recalling Bourgain's result, we need a few definitions:

Define $\Phi _\lambda f := \varphi_\lambda * f$, where $\varphi$  Schwartz function satisfying
\begin{equation}\label{e:1phi}
 \mathbf 1_{[-1/8,1/8]} \leq \hat{\varphi} \leq \mathbf 1_{[-1/4,1/4]}, 
\end{equation}  
and $ \varphi _{\lambda } (y) = \frac{1}{\lambda} \varphi (\frac{y}{\lambda} ) $.  
Next, let $ \{\theta_1, \dots, \theta_N \}$ be points in $ \mathbb R $ which are  $ \tau $-separated, in that 
$ \lvert  \theta _m - \theta _n\rvert> \tau >0 $ for $ m\neq n$.  Define a maximal operator by 
\begin{equation}\label{e:M}
Mf(x):= \sup_{ \lambda > 1/\tau } \left| \sum_{n=1}^N e(\theta_n x) \Phi _\lambda \big( \hat{f}(\cdot + \theta_n)\big)^{\vee}(x)  \right|. 
\end{equation}
Trivially, the operator norm of $ M$ is dominated by $ N$.  The key observation is that that 
this trivial bound can be improved to the much smaller term $ {\log}^2 N$.  
\begin{theorem}\cite[Lemma 4.13]{B3}\label{t:BMax}
For all $ N \geq 2$,  $ 0 < \tau < \infty $ and   $ \tau $-separated points  $ \{\theta_1, \dots, \theta_N \}$, there holds 
\[ \|Mf \|_{L^2(\RR)} \lesssim {\log}^2 N \|f\|_{L^2(\RR)}.\]
\end{theorem}

This inequality was extended in \cite{KL}, where the averages were replaced by 
oscillatory singular integrals. For $k_0 \geq 1$ arbitrary but fixed, define 
\begin{equation} \label{e:Tlambda}
T_{\lambda,d} f(x) = T_\lambda f(x) :=   \sum_{k_0 \leq k} \int e(-\lambda t^d) \psi_k(t) f(x-t) \; dt. 
\end{equation}
Above, $2 ^{ k_0} \geq \tau^{-C} $, and the dependence on $ k_0$ is uniform subject to this constraint.  
For $ \{\theta_1, \dots, \theta_N \}$ that are  $ \tau $-separated as in Theorem~\ref{t:BMax}, define 
\begin{equation}\label{e:T}
 T_d f(x) := \sup_{0< \lambda \leq 1 } \left| \sum_{n=1}^N e( \theta_n x) T_\lambda \left( \widehat{ \varphi _{\tau } } \hat{ f}( \cdot + \theta_n)  \right)^{\vee}(x) \right|.
\end{equation}
This definition matches that of \eqref{e:M}, except that there is an additional convolution with 
$  \varphi _{\tau } $ as in \eqref{e:1phi}.  
In \cite[Theorem 3.5]{KL}, an analogous multi-frequency estimate was proven for $T_2$.
\begin{theorem}
For all  $ 1 \leq \tau^{-C} \leq 2 ^{k_0} < \infty $,     $ N \geq 2$, and  $ \tau $-separated points  $\{\theta_1, \dots, \theta_N \}$, we have
\[ \| T_2 f\|_{L^2(\mathbb{R})} \lesssim {\log}^2 N \| f\|_{L^2(\mathbb{R})}.\]
\end{theorem}
By transference, Lemma \ref{trans}, the same result holds with $\ell^2$ norms replacing $L^2(\RR)$ norms.\footnote{Strictly speaking, Lemma \ref{trans} does not apply, since our set of modulation parameters is uncountable. But, by continuity we may restrict our set of modulation parameters to the rationals, at which point we may appeal to monotone convergence to apply Lemma \ref{trans}.}

Unfortunately, the arguments of \cite{KL} produce a polynomial-in-$N$ norm growth on $\ell^p$, which limits the utility of the operators $T_d$ for $p$ away from $2$ -- when $\{\theta_1,\dots,\theta_N\}$ are generic $\tau$-separated frequencies.

On the other hand, when the frequencies are replaced with sets $\mathcal{U}_N$, defined in Theorem \ref{Multtheorem}, we are able to enjoy sub-polynomial norm growth on $\ell^p, \ p \geq 2$. Before stating our theorem, we need modify the definition of $T_\lambda$ to contend with a truncation parameter:
\begin{equation}\label{Tredef}
T_\lambda f(x) :=   \sum_{k_0 \leq k} \int e(-\lambda t^d) \psi_k(t) f(x-t) \; dt \ \cdot \mathbf{1}_{\lambda \lesssim k^C 2^{-dk}} =: \sum_{k_0 \leq k} \int e(- \lambda t^d) \psi_k(t) f(x-t) \; dt \ \cdot \mathbf{1}_{\lambda \leq 2^{-c(k)} }
\end{equation}
i.e.\ the implicit constants in the statement $\lambda \lesssim k^C 2^{-dk}$ are chosen so that
\begin{equation}\label{impconsts}
C_k k^C 2^{-dk} = 2^{-c(k)}
\end{equation}
are dyadic. This means that for each $k$, the implicit constants $\{ C_k \}$ are fixed only up to a multiplicative factor of $2$, but this will not be a problem. We state our main theorem below.

\begin{theorem}\label{keyest}
For any $s \geq 1$, and any $0 < \rho \ll 1$, suppose $\chi_s(\beta) := \chi(D_s \beta)$ for some $D_s \geq 2^{2^{10s\rho}}$ (say), and consider 
\begin{equation}\label{e:Ts}
 T_{d,s}f(x) := T_s f(x) := \sup_{0 \leq \lambda \leq 1} \left| \sum_{\theta \in \mathcal{U}_{2^s} } e(\theta x) T_\lambda \left( \chi_s \hat{f}(\cdot+\theta) \right)^{\vee}(x) \right|.
\end{equation}
Then, for any $2 \leq p <\infty$,
\[ \| T_s f \|_{\ell^p} \lesssim_{p,\rho} s 2^{2s \rho} \| f\|_{\ell^p}.\]
\end{theorem}
\begin{remark}\label{cheaprem}
By the triangle inequality and Lemma \ref{trans}, a trivial estimate in $\ell^p$ is
\[ \| T_s f \|_{\ell^p} \lesssim_p |\mathcal{U}_{2^s}| \cdot \|f \|_{\ell^p} \lesssim_{\rho} 2^{c 2^{s \rho}} \| f\|_{\ell^p}.\]
Consequently, in proving this theorem we may assume that $s \gg \rho^{-1}$ is sufficiently large.
\end{remark}
The scheme of the proof of Theorem \ref{keyest} will follow that of \cite[\S 3]{KL}. We review this approach below.

First, though, we recall the following multi-frequency lemmas. These results essentially appear in \cite[\S 6-7]{MST1}.

One piece of notation. For each $j \geq 1$, let $A_j$ denote one of the following two convolution operators, with kernel given by either 
\[ \sum_{ k \geq j} {\psi_k} \ \text{ or } \ {\varphi_j} := 2^{-j} \varphi(2^{-j}\cdot ),\] for some (say) Schwartz function $\varphi$. Consider the maximal function,
\begin{equation}\label{A_s} \mathcal{A}_s g(x) := \sup_{j \geq 1} \left|\sum_{\theta \in \mathcal{U}_{2^s}} e(\theta x) A_j 
\left( \chi_s \hat{g}(\cdot+\theta) \right)^{\vee}(x) \right|,
\end{equation}
where $\chi_s$ is as above. Although these maximal functions depend on $\psi$ or on the particular choice of Schwartz function, the estimates for $\mathcal{A}_s$ are uniform (among appropriately normalized functions).
\begin{proposition}\label{MSTLp}
One has the following norm estimates on $\mathcal{A}_s$:
\[ \| \mathcal{A}_s g \|_p \lesssim_{\rho,p} s 2^{2 s \rho} \| g \|_{p}.\]
The implicit constant is uniform in $s$.
\end{proposition}
\begin{remark}
In \cite{MST1}, this result is proven with $\mathcal{U}_{2^s}$ replaced with $\mathcal{U}_{s^l}$ for some (sufficiently large) integer $l$, and with additional Weyl sums weighting the operator. These Weyl sums favorably contribute to the $\ell^2$ norm of the operator. In $\ell^p$, one is able to approximate these Weyl sums on
\[ \supp \bigcup_{\theta \in \mathcal{U}_{s^l}} \chi_s(\cdot - \theta) \]
by the symbol of an averaging operator, which acts as an $\ell^p$ multiplier (up to a logarithmic loss in $s^l$), see the proof of \cite[Theorem 6.2]{MST1}. The absence of these Weyl sums actually simplifies the argument, as no approximation on $\ell^p$ is needed. One uses the same splitting of scales as in \cite[\S 6-7]{MST1}: for small scales, one uses the Rademacher-Menshov style argument of \cite[Lemma 2.2]{MST1}. In the opposite case, where the least common multiple of the denominators of the frequencies in $\mathcal{U}_{2^s}$ is very small relative to the scale of the averaging operators, one uses periodicity. This splitting of scales is chosen according to whether the scale of the operator is greater or less than (say)
\[ 2^{\kappa_s}, \ \kappa_s := 2^{2\rho s}.\]
\end{remark}

We emphasize that in what follows, we will repeatedly rely on the fact that all frequencies appearing come from the sets $\mathcal{U}_{2^s}$. Our frequencies are therefore separated by
$2^{-c 2^{s \rho}},$ and we have localized to
$2^{ - 2^{ 10 \rho s }}$
balls around each frequency (which are much smaller for sufficiently large $s \gg \rho^{-1}$).

\subsection{Proof Overview}
The operator $ T _{\lambda }$ in \eqref{e:Tlambda} is decomposed as follows. In the integral 
\begin{equation*}
\int e(-\lambda t^d) \psi_k(t) f(x-t) \; dt, 
\end{equation*}
the variable $ t$ is approximately $ 2 ^{k}$ in magnitude.  And, we will decompose the operator and maximal function so that 
$ \lambda t ^{d} \approx \lambda 2^{dk}$ is approximately constant.  
Then, write 
\[ \aligned
 T_\lambda f(x)=&\sum_{l \in \Z}  \sum_{ k_0 \leq k  }  \int e(-\lambda t^d) \psi_k(t) f(x-t) \; dt \cdot \mathbf{1}_{\lambda \leq 2^{-c(k)}} \mathbf 1_{\{2^l \leq 2^{dk} \lambda < 2^{l+1}\}} \\
& \qquad =: \sum_{l \in \Z} \int  e(-\lambda t^d) \psi_{k(\lambda,l)}(t) f(x-t)  =: \sum_{l \in \Z} T_\lambda^l f(x). \endaligned \]
Here 
\[ \psi_{k(\lambda,l)} =\begin{cases} \psi_k &\mbox{if } 2^l \leq 2^{dk} \lambda < 2^{l+1}, \ l \leq dk - c(k)  \\ 
0 & \mbox{otherwise}. \end{cases} \]
Note that the condition $l \leq dk - c(k)$ is just a restatement of the condition
\[ 2^l \lesssim k^C,\]
where the implicit constants are $k$ dependent, but only vary by a multiplicative factor of $2$ (see \eqref{impconsts} above).

The supremum over $ \lambda $ is then divided into four separate cases, according to the relative size of 
$l$ and $ N$. For notational ease, we set
\[ \mathcal{U} := \mathcal{U}_{2^s} \]
in the remainder of this section.

\begin{align}
Tf &\leq  \sup_{0 < \lambda \leq 1} \left| \sum_{\theta \in \mathcal{U}} e(\theta x) \sum_{l < -C_{d,p} 2^{s \rho}} T_\lambda^l  \left( \chi_s \hat{f}(\cdot + \theta) \right)^{\vee}(x) 
\right| 
\label{e:1}\\
& \qquad + \sum_{l = -C_{d,p}  2^{s \rho} }^ 0 \sup_{0 < \lambda \leq 1} 
\left| \sum_{\theta \in \mathcal{U} } e(\theta x) T_\lambda^l \left( 
\chi_s \hat{f}(\cdot + \theta) \right)^{\vee} (x)
\right| 
\label{e:2}\\
& \qquad \qquad + \sum_{l = 1}^ {C_{d,p} 2^{s \rho}}  \sup_{0 < \lambda \leq 1} 
\left| \sum_{\theta \in \mathcal{U}} e(\theta x) T_\lambda^l \left( \chi_s \hat{f}(\cdot + \theta)  \right)^{\vee} (x)
\right| 
\label{e:3}\\
& \qquad \qquad \qquad+ \sum_{ l > C_{d,p} 2^{s \rho}} \sum_{ \theta \in \mathcal{U}} \sup_{0 < \lambda \leq 1} \left| T_\lambda^l \left( \chi_s \hat{f}(\cdot + \theta)  \right)^{\vee} (x)
\right|. 
\label{e:4}
\end{align}

We begin with the first Case, \eqref{e:1}. Indeed, in this regime, when
\[ l < - C_{d,p} 2^{ s \rho},\]
we automatically have $2^l \lesssim k^C$, so the additional truncations introduce in \eqref{Tredef} have no effect. Consequently, \eqref{e:1} may be treated as in \cite[p.12-13]{KL}. The point is that the phase is so small that the operators 
\[ \sum_{l < -C_{d,p} 2^{s \rho}} T_\lambda^l ``=" \sum_{ k \geq k_0 : \lambda 2^{dk} \leq 2^{-C_{d,p} 2^{s \rho} }} \psi_k,\]
up to error terms which are controlled by negligible multiples -- on the order of $|\mathcal{U}|^{-C}$ -- of the Hardy-Littlewood maximal function, which allow us to apply the triangle inequality and sum over each frequency individually. More precisely, 
\begin{equation}\label{errorTpsi}
\left| \sum_{l < -C_{d,p} 2^{s \rho}} T_\lambda^l f - \sum_{ k \geq k_0 : \lambda 2^{dk} \leq 2^{-C_{d,p} 2^{s \rho} }} \psi_k *f \right| \lesssim |\mathcal{U}|^{-C} P_{k'}(t)*|f|
\end{equation}
where $P(t)$ is a non-negative Schwartz function, with Fourier transform supported in a small ball near the origin, and
\begin{equation}\label{PP} P_{k'}(t) := 2^{-k'} P(2^{-k'}t)
\end{equation}
where $k' \geq k_0$ is the largest integer such that $2^{dk'}\lambda \lesssim 2^{-C_{d,p} 2^{s \rho}}$. By Lemma \ref{trans}, the maximal function associated to these convolution kernels is bounded on $\ell^p$ with norm $\lesssim |\mathcal{U}|^{-C}$, so we are free to sum over the $|\mathcal{U}|$-many distinguished frequencies and do away with the error term; the upshot is that, after appealing to Proposition \ref{MSTLp}, we may bound the $\ell^p$ norm of \eqref{e:1} by $s 2^{2 s\rho}$. 

The final Case \eqref{e:4} simply follows from the following special case of Stein-Wainger \cite{SW}; the point is that as $l \gg_{d,p} 2^{s \rho}$, we may simply use the triangle inequality to trivially absorb the sum over $| \mathcal{U}| \lesssim 2^{c 2^{s \rho}}$ many frequencies.

\begin{lemma}\label{l:4} For any $ l \geq 0$, and any $1 < p < \infty$, for some $\delta_{d,p} > 0$, there holds 
 \begin{equation}\label{e:4<}
 \bigl\| \sup_{0 < \lambda \leq 1}  |T_\lambda^l f| \bigr\|_{p} \lesssim 2^{- \delta_{d,p} l} \| f\|_{p}.  
 \end{equation}
 \end{lemma}
 %%%%%%%%%%%%%%%%%%%%%%%%%%%%%% LEMMA LEMMA LEMMA
The proof of  Stein-Wainger \cite[Theorem 1]{SW} contains this result without the truncation parameters, but the changes introduced are formal; in fact, this result can be proven directly by a straightforward $TT^*$ argument (which yields an estimate of $\delta_{d,2} \approx \frac{1}{d}$) and trivial interpolation.

Consequently, by Lemma \ref{trans}, we may similarly bound
\[ 
\bigl\| \sup_{0 < \lambda \leq 1}  |T_\lambda^l \left( \chi_s \hat{f}(\cdot + \theta) \right)^{\vee} | \bigr\|_{\ell^p} \lesssim 2^{- \delta_{d,p} l} \| f\|_{\ell^p} \] 
for any $\theta$.
Since we have chosen $l \gg_{d,p} 2^{s \rho}$, we may sum over $|\mathcal{U}| \lesssim 2^{c 2^{s \rho}}$ many frequencies to estimate
\[  
\| \eqref{e:4} \|_{\ell^p} \lesssim \sum_{l > C_{d,p} 2^{s \rho}} |\mathcal{U}| \cdot 2^{- \delta_{d,p} l} \|f \|_{\ell^p} \lesssim \|f \|_{\ell^p}, \]
for a sufficiently large choice of $C_{d,p}$.

In particular, the main effort boils down to bounding \eqref{e:2} -- the \emph{``stationary'' critical regime} -- and \eqref{e:3} -- the \emph{``oscillatory'' critical regime} --  in $\ell^p$. We accomplish this in the following subsections.

\subsection{Cases Two and Three: $|l| \leq C_{d,p} 2^{s \rho}$}
We turn to the most technical part of the paper. The result we will establish is the following theorem. 
\begin{theorem}\label{MAINEST}
For any $2 \leq p < \infty$, the following estimates hold (with implicit constant independent of $s$):
\[ \| \eqref{e:2} \|_{\ell^p} + \| \eqref{e:3} \|_{\ell^p} \lesssim_{\rho,p} s 2^{2s\rho} \|f\|_{\ell^p}.\]
\end{theorem}

The terms to control are \eqref{e:2} and \eqref{e:3}, in which the sum over $ l$ is limited to  $|l| \leq C_{d,p} 2^{s \rho}$. 

Much of the argument is common to both cases of $ l \leq 0$ and $ l > 0$. 
First, some notation:

For $k = k(\lambda,l)$ as above, define
\begin{equation}\label{mu}
\mu(\lambda, l) := \int e(-\lambda t^d) \psi_k(t) \ dt = \int e(-\lambda 2^{kd} t^d) \psi(t) \ dt;
\end{equation}
by the mean-value theorem -- taking into account the mean-zero nature of $\psi$ -- and the principle of non-stationary phase, one may estimate
\[ |\mu(\lambda,l) | \lesssim_N \min\{ 2^l, 2^{-lN} \}.\]
We will also let
\begin{equation}\label{mubar}
\overline{\mu}(\lambda, l) :=  -2\pi i \cdot 2^{-k} \int e(-\lambda t^d) t \psi_k(t) \ dt = 
- 2\pi i \cdot \int e(-\lambda 2^{kd} t^d) t \psi(t) \ dt .
\end{equation}
Once again, we may estimate
\[ |\overline{\mu}(\lambda,l) | \lesssim_N \min\{ 2^l, 2^{-lN} \}.\]

Let $\Theta$ be as in \eqref{Theta},
and set
\begin{align}
\Theta_<(\xi) &:= \sum_{j \geq 1} \Theta_j(\xi)
\label{e:Theta<}\\
\Theta_>(\xi) &:= \sum_{j \leq 0} \Theta_j(\xi);
\label{e:Theta>}
\end{align}
we will use this splitting in the ``stationary'' regime, 
\[ -C_{d,p} 2^{s\rho} \leq l \leq 0.\]
In the more ``oscillatory'' regime, 
\[ 1 \leq l \leq C_{d,p} 2^{s\rho},\]
the splitting we need is slightly more involved:
\begin{align}
\zeta(\xi) &:= \sum_{|j| \leq C} \Theta_j(\xi)
\label{e:zeta}\\
\Theta_L(\xi) &:= \sum_{j > C} \Theta_j(\xi), \text{ and }
\label{e:ThetaL}\\
\Theta_H(\xi) &:= \sum_{j < -C} \Theta_j(\xi)
\label{e:ThetaH}
\end{align}
for some sufficiently large $C \gg 1$. 

Next, define the function
\begin{equation}\label{g}
g_\lambda(t) := e(-\lambda t^d) \psi_{k}(t) 
\end{equation}
where $k = k(\lambda,l)$. 
%In particular, we may assume throughout that \[ 2^l \lesssim k^C.\]
With these preliminaries in mind, we turn to the ``stationary'' regime.

\subsection{The ``Stationary'' Critical Regime: $-C_{d,p} 2^{s \rho} \leq l \leq 0$}
We consider the maximal function
\[  
\mathcal{M}f(x) := \sup_k |\mathcal{M}_k f| :=
\sup_k \sup_{2^{l-dk} \leq \lambda < 2^{l-dk+1}} \left| \sum_{\theta \in \mathcal{U}} e(\theta x) g_\lambda * ( \chi_s \hat{f}(\cdot + \theta))^{\vee}(x) \right |. \]
We decompose our maximal functions $\mathcal{M}_k$ further. We dominate
\[ \mathcal{M}_k f \leq \mathcal{L}_k f + \mathcal{H}_k f,\]
where
\[ \aligned 
\mathcal{L}_kf &:= \sup_{2^{l-dk} \leq \lambda < 2^{l-dk+1}} \left| \sum_{\theta \in \mathcal{U}} e(\theta x) \left( \widehat { g_\lambda } \Theta_{<}(2^{k} \cdot) \right)^{\vee} * ( \chi_s \hat{f}(\cdot + \theta))^{\vee}(x) \right |, \ \text{ and} \\
\mathcal{H}_kf &:= 
\sup_{2^{l-dk} \leq \lambda < 2^{l-dk+1}} \left| \sum_{\theta \in \mathcal{U}} e(\theta x) \left( \widehat { g_\lambda } \Theta_{>}(2^{k} \cdot) \right)^{\vee} * ( \chi_s \hat{f}(\cdot + \theta))^{\vee}(x) \right |, \endaligned \]
where $\Theta_{<}, \Theta_{>}$ are defined in \eqref{e:Theta<} and \eqref{e:Theta>} respectively.
We will first prove $\ell^p$ estimates on $\mathcal{L}_k$; up to error terms that are controlled by Bourgain's maximal function, we will see that $\sup_k |\mathcal{L}_kf|$ can essentially be dominated by \emph{vector-valued multi-frequency Mikhlin multipliers;} $\sup_k |\mathcal{H}_kf|$ will be estimated similarly, though is simpler to handle, as ``zero-frequency'' considerations do not arise. 
\subsubsection{Estimating $\sup_k |\mathcal{L}_k f|$ and $\sup_k |\mathcal{H}_k f|$}
Estimates for \eqref{e:2} will follow from the following two propositions.

\begin{proposition}\label{Lest0}
For any $1 < p < \infty$, we have the following estimate:
\[ \| \sup_k |\mathcal{L}_k f| \|_{\ell^p} \lesssim_{p,\rho} 2^{l/2} \times s 2^{2 s \rho} \|f\|_{\ell^p}.\] 
\end{proposition}

\begin{proposition}\label{Hest0}
For any $1 < p < \infty$, we have the following estimate:
\[ \| \sup_k |\mathcal{H}_k f| \|_{\ell^p} \lesssim_{p,\rho} 2^{l/2} \times s \|f\|_{\ell^p}.\] 
\end{proposition}

We begin with Proposition \ref{Lest0}, which will require a further decomposition of each maximal function $\mathcal{L}_k$.

The first order of business is to replace $\mathcal{L}_kf$ with
\[ \mathcal{L}_kf \leq \mathcal{L}_k^1 f +\mathcal{L}_k^2 f,\]
where 
\[ \aligned 
\mathcal{L}_k^1 f &=: \sup_{2^{l-dk} \leq \lambda < 2^{l-dk+1}} \left| \sum_{\theta \in \mathcal{U}} e(\theta x) \left( \Big( \widehat { g_\lambda } - \mu(\lambda,l) \Big) \Theta_{<}(2^{k} \cdot) \right)^{\vee} * ( \chi_s \hat{f}(\cdot + \theta))^{\vee}(x) \right |, \ \text{ and} \\
\mathcal{L}_k^2 f&:=
\sup_{2^{l-dk} \leq \lambda < 2^{l-dk+1}} \left| \mu(\lambda,l) \times \sum_{\theta \in \mathcal{U}} e(\theta x) \left( \Theta_{<}(2^{k} \cdot) \right)^{\vee} * ( \chi_s \hat{f}(\cdot + \theta))^{\vee}(x) \right |. \endaligned\]
Here, $\mu(\lambda,l)$ is defined in \eqref{mu}, and $g_\lambda$ is defined in \eqref{g}.

Now,
\[ \sup_k |\mathcal{L}_k^2 f| \lesssim 2^{l} \mathcal{A}_s f,\]
where $\mathcal{A}_s$ is Bourgain's maximal function, defined in \eqref{A_s}. Consequently, we have the acceptable estimate
\[ \| \sup_k |\mathcal{L}_k^2 f| \|_{\ell^p} \lesssim_{p,\rho} 2^{l} \cdot s 2^{2 s \rho} \| f\|_{\ell^p}.\]
With $\overline{\mu}(\lambda,l)$ defined in \eqref{mubar}, we now expand 
\[ \aligned 
&\Theta_{<}(2^{k} \xi) \times \Big( \widehat{g_\lambda}(\xi) - \mu(\lambda,l) \Big) \\
& \qquad = \sum_{m \geq 1} \; \Theta_{k + m}(\xi) \times \Big(  \widehat{g_\lambda}(\xi) - \mu(\lambda,l) \Big) \\
& \qquad \qquad = \sum_{m \geq 1} \; \Theta_{k + m}(\xi) \times \left( \widehat{g_\lambda}(\xi)
- \mu(\lambda,l)  - 2^{-m} \cdot \overline{\mu}(\lambda,l) \times \int_{-\infty}^{2^{k+m} \xi} \overline{\Theta}(t) \ dt \right) \\
& \qquad \qquad \qquad +
\sum_{m \geq 1} \; 2^{-m} \overline{\mu}(\lambda,l) \times \left( \Theta_{k + m}(\xi) \times \int_{-\infty}^{2^{k+m} \xi} \overline{\Theta}(t) \ dt \right). \endaligned \]
Setting
\[ M^{k,l,m}(\lambda,\xi) := \Theta_{k + m}(\xi) \times \left( \widehat{g_\lambda}(\xi)
- \mu(\lambda,l)  - 2^{-m} \overline{\mu}(\lambda,l) \times \int_{-\infty}^{2^{k+m} \xi} \overline{\Theta}(t) \ dt \right) \]
and
\begin{equation}\label{tT}
\tilde{\Theta}(\xi) := \Theta(\xi) \cdot \int_{-\infty}^\xi \overline{\Theta}(t) \ dt,
\end{equation}
we may dominate
\[ \sup_k |\mathcal{L}^1_k f| \leq \sum_{m \geq 1} \sup_k |\mathcal{L}_k^{3,m} f| + \sum_{m \geq 1} \sup_k |\mathcal{L}_k^{4,m} f|,\]
where
\[ 
\mathcal{L}_k^{3,m} f(x) =:
\sup_{2^{l-dk} \leq \lambda < 2^{l-dk+1}}
\left| \sum_{\theta \in \mathcal{U}} e(\theta x) \mathcal{F}_{\xi}^{-1} \big( M^{k,l,m}(\lambda, \cdot) \big) * ( \chi_s \hat{f}(\cdot + \theta))^{\vee}(x) \right |, \]
and
\[ \aligned 
&\mathcal{L}_k^{4,m} f(x) \\
& \qquad := \sup_{2^{l-dk} \leq \lambda < 2^{l-dk+1}} 
\left| 2^{-m}  \cdot \overline{\mu}(\lambda,l) \times 
\sum_{\theta \in \mathcal{U}} e(\theta x) 
\big( \tilde{\Theta}(2^{k+m} \cdot) \big)^{\vee} * ( \chi_s \hat{f}(\cdot + \theta))^{\vee}(x) \right |. \endaligned \]
But now we observe that
\[ \sup_k |\mathcal{L}_k^{4,m} f| \lesssim_N 2^{l-m} \mathcal{A}_s f, \]
see \eqref{A_s}, so we may sum in $m \geq 1$, and discard this contribution. We are left to estimate
\[ \sup_k |\mathcal{L}_k^{3,m} f| \]
in $\ell^p$. To proceed we record the following estimates on the multipliers $M^{k,l,m}(\lambda,\xi)$.
\begin{lemma}\label{MIKEST1}
For any $\lambda$, the following estimates hold.
\[ |M^{k,l,m}(\lambda, \xi)| + 2^{-k} |\partial_\xi M^{k,l,m}(\lambda, \xi)| \lesssim 2^{-m} \cdot \mathbf{1}_{|\xi| \approx 2^{-k-m}}. \]
Moreover, the same estimates are satisfied uniformly in $\lambda$, by
\[ 2^{-dk} \partial_\lambda M^{k,l,m}(\lambda,\xi).\]
\end{lemma}
\begin{proof}
For the estimate without the derivative, we just use the mean value theorem
\begin{equation}\label{MVT}
\left| \int e(-\lambda 2^{kd} t^d - 2^k \xi t) \psi(t) \ dt - \mu(\lambda,l) \right| \leq \int |e(-2^k \xi t) - 1| |\psi(t)| \ dt \lesssim 2^k |\xi|,
\end{equation}
and the trivial estimate,
\[ \int_{-\infty}^{2^{k+m} \xi} |\overline{\Theta}(t) | \lesssim 1 \]
for $|\xi| \approx 2^{-k-m}$.
The estimate with the derivative estimate follows from the same mean value theorem argument of \eqref{MVT}. The final point is trivial; the key point is that $t \mapsto \frac{t^d}{2^{dk}} \psi_k(t)$ satisfies (up to harmless constants) the same differential estimates as does $\psi_k(t)$.
\end{proof}

Motivated by these stationary phase calculations, we bound
\[ |\mathcal{L}_k^{3,m} f| \leq (\mathcal{S}_k^{3,m,0} f)^{1/2} \cdot (\mathcal{S}_k^{3,m,1} f)^{1/2} \lesssim 2^{dk/2} \cdot \mathcal{S}_k^{3,m,0} f + 2^{-dk/2} \cdot \mathcal{S}_k^{3,m,1} f,\]
where
\[ |\mathcal{S}_k^{3,m,i} f|^2 := 
\int_{2^{l-dk}}^{2^{l-dk+1}} 
\left| \sum_{\theta \in \mathcal{U}} e(\theta x) \mathcal{F}_{\xi}^{-1} \big( \partial_\lambda^i M^{k,l,m}(\lambda, \cdot) \big) * ( \chi_s \hat{f}(\cdot + \theta))^{\vee}(x) \right |^2 \ d\lambda, \ i = 0,1.\]
We will estimate 
\[ 2^{dk/2} \cdot \mathcal{S}_k^{3,m,0}f =: \mathcal{S}_k^{m} f,\]
as the other term can be treated similarly. We now replace,
\[ \sup_k |\mathcal{S}_k^mf| \leq \left( \sum_k |\mathcal{S}_k^m f|^2 \right)^{1/2};\]
by Theorem \ref{Multtheorem} and Proposition \ref{VVMT}, we have
\[ \| \left( \sum_k |\mathcal{S}_k^m f|^2 \right)^{1/2} \|_{\ell^p} \lesssim s \cdot C(m,l) \|f\|_{\ell^p},\]
where
\[ \aligned 
C(m,l)^2 &:= \sup_\xi \; \sum_k 2^{dk} \int_{2^{l-dk}}^{2^{l-dk+1}} |M^{k,l,m}(\lambda,\xi)|^2 \ d\lambda \\
& \qquad \qquad + \sup_\xi \; \sum_k 2^{dk} \int_{2^{l-dk}}^{2^{l-dk+1}} |\xi|^2 |\partial_\xi M^{k,l,m}(\lambda,\xi)|^2 \ d\lambda \\
& \qquad \qquad \qquad \qquad \lesssim 2^{l-2m} \endaligned \]
by Lemma \ref{MIKEST1}. Putting everything together yields Proposition \ref{Lest0}.

The proof of Proposition \ref{Hest0} follows a similar strategy. 
The key estimates are 
\begin{equation}\label{Fdecayest}
|\widehat{g_\lambda}(\xi) \Theta_{>}(2^{k} \xi)| +
2^{-k} | \partial_\xi \left( \widehat{g_\lambda}(\xi) \Theta_{>}(2^{k} \xi) \right)| \lesssim_N 
(2^k |\xi|)^{-N} \cdot \mathbf{1}_{|\xi| \gg 2^{-k}}
\end{equation}
by the principle of non-stationary phase.
Consequently, the square function
\[ \left( \sum_k 2^{dk} \int_{2^{l-dk}}^{2^{l-dk+1}} \left| \left( \widehat{g_\lambda}(\xi) \Theta_{>}(2^{k} \xi) \hat{f}(\xi) \right)^{\vee}\right|^2 \ d\lambda \right)^{1/2} \]
is a vector-valued Mikhlin multiplier, with norm $\lesssim 2^{l/2}$:
\[ \| \left( \sum_k 2^{dk} \int_{2^{l-dk}}^{2^{l-dk+1}} \left| \left( \widehat{g_\lambda}(\xi) \Theta_{>}(2^{k} \xi) \hat{f}(\xi) \right)^{\vee} \right|^2 \ d\lambda \right)^{1/2} \|_p \lesssim_p 2^{l/2} \| f\|_p,\]
which completes the proof upon an application of Theorem \ref{Multtheorem}.

\subsection{The ``Oscillatory'' Critical Regime: 
$1 \leq l \leq C_{d,p} 2^{s \rho}$}
Recalling that $k^C \gtrsim 2^l$, we decompose our maximal functions $\mathcal{M}_k$ as a sum of three terms,
\[ \mathcal{M}_k f \leq \mathcal{L}_k f + \mathcal{Z}_k f + \mathcal{H}_k f,\]
where here
\[ \aligned 
\mathcal{L}_kf &:= \sup_{2^{l-dk} \leq \lambda < 2^{l-dk+1}} \left| \sum_{\theta \in \mathcal{U}} e(\theta x) \left( \widehat { g_\lambda } \Theta_L(2^{k-l} \cdot) \right)^{\vee} * ( \chi_s \hat{f}(\cdot + \theta))^{\vee}(x) \right | \\
\mathcal{Z}_kf &:= \sup_{2^{l-dk} \leq \lambda < 2^{l-dk+1}} \left| \sum_{\theta \in \mathcal{U}} e(\theta x) \left( \widehat { g_\lambda } \zeta(2^{k-l} \cdot) \right)^{\vee} * ( \chi_s \hat{f}(\cdot + \theta))^{\vee}(x) \right |, \ \text{ and} \\
\mathcal{H}_kf &:= 
\sup_{2^{l-dk} \leq \lambda < 2^{l-dk+1}} \left| \sum_{\theta \in \mathcal{U}} e(\theta x) \left( \widehat { g_\lambda } \Theta_H(2^{k-l} \cdot) \right)^{\vee} * ( \chi_s \hat{f}(\cdot + \theta))^{\vee}(x) \right |, \endaligned \]
with $\zeta, \Theta_L$ and $\Theta_H$ defined in \eqref{e:zeta}, \eqref{e:ThetaL}, and \eqref{e:ThetaH}.

As in the ``stationary'' regime, up to Bourgain-controlled errors, we can estimate $\sup_k |\mathcal{L}_kf|$ and $\sup_k |\mathcal{H}_kf|$ using singular integral techniques; the heart of the problem lies in estimating $\sup_k |\mathcal{Z}_k f|$, where singular integral techniques are ineffective -- essentially due to the fact that
\[ \xi \mapsto \widehat{g_\lambda}(\xi) \zeta(2^{k-l}\xi) = \widehat{ G_\lambda }(\xi),\]
with $G_\lambda$ defined in \eqref{invF} above, has an unacceptably large Mikhlin multiplier norm of $2^{l/2}$. Rather, estimates for $\sup_k |\mathcal{Z}_k f|$ will follow (quickly) from our the square function estimates from \S \ref{s:SFE}; we will dominate $\sup_k |\mathcal{Z}_k f|$ by multi-frequency analogues of the square functions treated in Theorem \ref{SINGFREQEST}.

\subsubsection{Estimating $\sup_k |\mathcal{L}_k f|$ and $\sup_k |\mathcal{H}_k f|$}
In this section, we will reduce \eqref{e:3} to estimating $\sup_k |\mathcal{Z}_k f|$ by establishing the following two propositions.

\begin{proposition}\label{Lest}
For any $1 < p < \infty$, we have the following estimate:
\[ \| \sup_k |\mathcal{L}_k f| \|_{\ell^p} \lesssim_{p,\rho,N} 2^{-l N} \times s 2^{2 s \rho} \|f\|_{\ell^p}.\] 
\end{proposition}

\begin{proposition}\label{Hest}
For any $1 < p < \infty$, we have the following estimate:
\[ \| \sup_k |\mathcal{H}_k f| \|_{\ell^p} \lesssim_{p,\rho,N} 2^{-l N} \times s \|f\|_{\ell^p}.\] 
\end{proposition}

We begin with Proposition \ref{Lest}, which will require a decomposition of each maximal function $\mathcal{L}_k$. By arguing as in the ``stationary'' regime, we may dominate\footnote{The gain below comes from our estimates on $\mu(\lambda,l), \ \overline{\mu}(\lambda,l)$}
\[ \sup_k |\mathcal{L}_k f| \lesssim_N 2^{-lN} \mathcal{A}_{s} f + \sum_{m > C} \sup_k |\mathcal{L}_k^{m} f|,\]
where 
\[ \aligned 
\mathcal{A}_{s}f &:= \sup_k \left| \sum_{\theta \in \mathcal{U}} e(\theta x) \left( \Theta(2^k \cdot) \right)^{\vee} * \left( \chi_s \hat{f}(\cdot+\theta) \right)^{\vee}(x) \right| \\
& \qquad + \sup_k \left| \sum_{\theta \in \mathcal{U}} e(\theta x) \left( \tilde{\Theta}(2^k \cdot) \right)^{\vee} * \left( \chi_s \hat{f}(\cdot+\theta) \right)^{\vee}(x) \right| 
\endaligned
\]
is a sum of two maximal functions as in \eqref{A_s}, with $\tilde{\Theta}$ defined in \eqref{tT} above,
and $\mathcal{L}_k^m$ are defined below:
\[ 
\mathcal{L}_k^{m} f(x) =:
\sup_{2^{l-dk} \leq \lambda < 2^{l-dk+1}}
\left| \sum_{\theta \in \mathcal{U}} e(\theta x) \mathcal{F}_{\xi}^{-1} \big( M^{k,l,m}(\lambda,\cdot) \big) * ( \chi_s \hat{f}(\cdot + \theta))^{\vee}(x) \right |, \]
where in this regime we have
\[ M^{k,l,m}(\lambda,\xi) := 
\Theta_{k-l + m}(\xi) \times \left( \widehat{g_\lambda}(\xi)
- \mu(\lambda,l)  - 2^{l-m} \overline{\mu}(\lambda,l) \times \int_{-\infty}^{2^{k-l+m} \xi} \overline{\Theta}(t) \ dt \right).\]
To proceed we record the following estimates on the multipliers $M^{k,l,m}(\xi,\lambda)$.
\begin{lemma}\label{MIKEST2}
For any $\lambda$, the following estimates hold.
\[ |M^{k,l,m}(\lambda,\xi)| + 2^{-k} |\partial_\xi M^{k,l,m}(\lambda,\xi)| \lesssim_N 2^{-m-lN} \cdot \mathbf{1}_{|\xi| \approx 2^{l-k-m}}. \]
Moreover, the same estimates are satisfied uniformly in $\lambda$, by
\[ 2^{-dk} \partial_\lambda M^{k,l,m}(\lambda,\xi).\]
\end{lemma}
\begin{proof}
We begin again by estimating the term without the derivative:
\[ \left| \int_0^{2^k \xi} \left( \int e(\lambda 2^{dk} t^d - s t) t\psi(t) \ dt \right) \ ds \right| \lesssim_N 2^{k-lN} |\xi|,\]
since the bracketed expression is $\lesssim_N 2^{-lN}$ in magnitude by the principle of non-stationary phase; the term involving $\overline{\mu}(\lambda,l)$ can be estimated as in the proof of Lemma \ref{MIKEST1}. The derivative estimate follows similarly, and the final point is straightforward, as per Lemma \ref{MIKEST1}.
\end{proof}
As in the stationary case, we now replace
\[ |\mathcal{L}_k^{m} f| \leq (\mathcal{S}_k^{m,0} f)^{1/2} \cdot (\mathcal{S}_k^{m,1} f)^{1/2} \lesssim 2^{dk/2} \cdot \mathcal{S}_k^{m,0} f + 2^{-dk/2} \cdot \mathcal{S}_k^{m,1} f,\]
where
\[ |\mathcal{S}_k^{m,i} f|^2 := 
\int_{2^{l-dk}}^{2^{l-dk+1}} 
\left| \sum_{\theta \in \mathcal{U}} e(\theta x) \mathcal{F}_{\xi}^{-1} \big( \partial_\lambda^i M^{k,l,m}(\lambda,\cdot) \big) * ( \chi_s \hat{f}(\cdot + \theta))^{\vee}(x) \right |^2 \ d\lambda , \ i = 0,1.\]
We will estimate 
\[ 2^{dk/2} \cdot \mathcal{S}_k^{m,0}f =: \mathcal{S}_k^{m} f,\]
as the other term can be treated similarly. We again majorize
\[ \sup_k |\mathcal{S}_k^mf| \leq \left( \sum_k |\mathcal{S}_k^m f|^2 \right)^{1/2};\]
by Theorem \ref{Multtheorem} and Proposition \ref{VVMT}, we again have
\[ \| \left( \sum_k |\mathcal{S}_k^m f|^2 \right)^{1/2} \|_{\ell^p} \lesssim s \cdot C(m,l) \|f\|_{\ell^p},\]
where
\[ \aligned 
C(m,l)^2 &:= \sup_\xi \; \sum_k 2^{dk} \int_{2^{l-dk}}^{2^{l-dk+1}} |M^{k,l,m}(\lambda,\xi)|^2 \ d\lambda \\ 
& \qquad \qquad + 
\sup_\xi \; \sum_k 2^{dk} \int_{2^{l-dk}}^{2^{l-dk+1}} |\xi|^2 |\partial_\xi M^{k,l,m}(\lambda,\xi)|^2 \ d\lambda \\
& \qquad \qquad \qquad \qquad \lesssim_N 2^{-m -lN} \endaligned \]
by Lemma \ref{MIKEST2}, which yields Proposition \ref{Lest}.

The proof of Proposition \ref{Hest} follows a similar strategy. 
The key estimates are 
\begin{equation}\label{Fdecayest}
|\widehat{g_\lambda}(\xi) \Theta_H(2^{k-l} \xi)| +
2^{-k} | \partial_\xi \left( \widehat{g_\lambda}(\xi) \Theta_H(2^{k-l} \xi) \right)| \lesssim_N 
(2^k |\xi|)^{-N} \cdot \mathbf{1}_{|\xi| \gg 2^{l-k}},
\end{equation}
which follow from the principle of non-stationary phase.
Consequently, the square function
\[ \left( \sum_k 2^{dk} \int_{2^{l-dk}}^{2^{l-dk+1}} \left| \left( \widehat{g_\lambda}(\xi) \Theta_H(2^{k-l} \xi) \hat{f}(\xi) \right)^{\vee} \right|^2 \ d\lambda \right)^{1/2} \]
is a vector-valued Mikhlin multiplier, with norm $\lesssim 2^{-lN}$:
\[ \| \left( \sum_k 2^{dk} \int_{2^{l-dk}}^{2^{l-dk+1}} \left| \left( \widehat{g_\lambda}(\xi) \Theta_H(2^{k-l} \xi) \hat{f}(\xi) \right)^{\vee}\right|^2 \ d\lambda \right)^{1/2} \|_p \lesssim_p 2^{-lN} \| f\|_p,\]
which completes the proof upon an application of Theorem \ref{Multtheorem}.

Finally, we turn to $\sup_k|\mathcal{Z}_kf|$; the estimates of \S \ref{s:SFE} allow us to quickly dispose of this term.
\subsubsection{Estimating $\sup_k|\mathcal{Z}_kf|$}
Theorem \ref{keyest} will now follow from the following proposition.
\begin{proposition}
For any $2 \leq p < \infty$,
\[ \| \sup_k |\mathcal{Z}_k f| \|_{\ell^p} \lesssim_{p} l \cdot s \| f\|_{\ell^p} \lesssim_{p,d} s \cdot 2^{s \rho} \| f \|_{\ell^p}.\]
\end{proposition}
\begin{proof}
Substituting
\[ G_\lambda :=  g_\lambda * \zeta(2^{k-l} \cdot)^{\vee} \]
One may dominate
\[ \sup_k |\mathcal{Z}_k f| \lesssim \mathcal{S}_Z^1 f + \mathcal{S}_Z^2 f,\]
where
\[ \mathcal{S}_Z^1 f := \left( \sum_k 2^{dk} \int_{2^{l-dk}}^{2^{l-dk +1}} \left| \sum_{\theta \in \mathcal{U}} e(\theta x) G_\lambda* ( \chi_s \hat{f}(\cdot + \theta) )^{\vee} \right|^2 \ d\lambda \right)^{1/2},\]
and
\[ \mathcal{S}_Z^2 f := \left( \sum_k 2^{-dk} \int_{2^{l-dk}}^{2^{l-dk +1}} \left| \sum_{\theta \in \mathcal{U}} e(\theta x) (\partial_\lambda G_\lambda)* ( \chi_s \hat{f}(\cdot + \theta) )^{\vee} \right|^2 \ d\lambda \right)^{1/2} \]
is similar. The result now follows by Theorem \ref{Multtheorem} and Proposition \ref{SINGFREQEST}.
\end{proof}

This concludes the proof of Theorem \ref{keyest}. After various number-theoretic reductions, we will apply Theorem \ref{keyest} in \S \ref{s:comp} below to prove the $p \geq 2$ case of Theorem \ref{MAIN}.

\section{Most Modulation Parameters are Safe: A $TT^*$ Argument}\label{s:TT*1}
We begin this section by introducing the following sets of modulation parameters for each $j \geq 1$:
\begin{equation}\label{X_j}
X_j := 
\left\{ \frac{A}{Q} : (A,Q) = 1, \ Q \lesssim_{d,p} j^{C_{d,p}} \right\} + \{ |\beta| \lesssim_{d,p} j^{C_{d,p}} 2^{-dj} \}.
\end{equation}
Here, the sum denotes the Minkowski sum. We will always choose the second implicit constant in \eqref{X_j} so that
\begin{equation}\label{c(j)}
\{ |\beta| \lesssim_{d,p} j^{C_{d,p}} 2^{-dj} \} = [ -2^{-c(j)}, 2^{-c(j)} ]; 
\end{equation}
in particular, this implicit constants may vary by a multiplicative factor bounded by $2$ as $j$ changes (see \eqref{impconsts} above).

The main result of this section is the following theorem.

\begin{theorem}\label{NoX_j}
For any $1<p<\infty$, if the constant $C_{d,p}$ is chosen sufficiently large, for all integers $j \geq 1$,
\begin{equation}\label{M_jrest}
\| \sup_{\lambda \notin X_j} \left| \sum_{m} f(x-m) \psi_j(m) e(-\lambda m^d) \right| \|_{\ell^p} \lesssim j^{-2}.
\end{equation}
\end{theorem}
The remainder of this section will be taken with the proof of Theorem \ref{NoX_j}.

\subsection{The Set-up}
In what follows, for notational ease we will suppress all dependence on $d,p$ in our implicit constants. We begin by observing that since \eqref{M_jrest} is trivially bounded by $M_{HL}$, by interpolation it suffices to establish the conclusion of Theorem \ref{NoX_j} on $\ell^2$, with a decay factor of $j^{-C}$ instead of $j^{-2}$ for some $C$ sufficiently large. And so, we work only at an $\ell^2$ level, which will allow us to use the method of $TT^*$.

To this end, consider the kernel
\[ K_j(x,n) := \sum_m \psi_j(x-m) \psi_j(n-m) e( \lambda(x)(x-m)^d - \mu(n) (n-m)^d) \]
where $\lambda,\mu: \Z \to [0,1]$ are arbitrary functions. By the compact support of $\psi_j$, we may assume without loss of generality that $|x|, |n| \lesssim 2^j$.

We claim that unless $\lambda,\mu \in X_j$, there exist two sets, 
\[ E(x), E(n) \subset \{ |m| \lesssim 2^j\} \]
each with cardinality $\lesssim j^{-C} 2^j$, such that
\begin{equation}\label{TT*goal}
 |K_j(x,n)| \lesssim j^{-C} 2^{-j} \mathbf{1}_{|x-n| \lesssim 2^j} + 2^{-j} \mathbf{1}_{E(x)}(n) + 2^{-j} \mathbf{1}_{E(n)}(x).
\end{equation}
In particular,
\[ \aligned 
&\sum_{x,n} |g(x)| |K_j(x,n)| |f(n)| \\
& \qquad \lesssim j^{-C} \sum_x |g(x)| M_{HL}f(x) + 
\sum_n M_jg(n) |f(n)| + \sum_{x} |g(x)| M_{j} f(x), \endaligned \]
where
\begin{equation} 
M_j h(x) := \sup_{E} \
2^{-j} \cdot \mathbf{1}_{E}*h(x),
\end{equation}
where the supremum runs over all $E$ such that
\[ |E| \lesssim j^{-C} 2^j, \ E \subset \{ |m| \lesssim 2^j\}. \]
Note that this operator has $\ell^\infty$ operator norm $\lesssim j^{-C}$ and $\ell^1$ operator norm $1$ -- and thus $\ell^2$ norm $\lesssim j^{-C}$ as well.\footnote{To the best of the author's knowledge, a ``small-set'' maximal function to control kernels arising from $TT^*$ calculations was first used in \cite{SW}.}
Since, for an appropriate $\lambda,\mu$
\[ \sum_n K_j(x,n) f(n) = TT^* f(x) \]
for $T$ a linearization of the supremum in \eqref{M_jrest}, we are able to conclude favorable $\ell^2$ estimates for this maximal operator.

In what follows, we shall regard $n$ as fixed, and will prove that if the set
\[ E(n) := \{ |x| \lesssim 2^j : |K_j(x,n)| \gtrsim j^{-C} 2^{-j} \} \]
has cardinality greater than $j^{-C} 2^j$, then $\mu(n) \in X_j$. A symmetric argument will apply to the sets
\[ E(x) := \{ |n| \lesssim 2^j : |K_j(x,n)| \gtrsim j^{-C} 2^{-j} \}. \]

\subsection{Exponential sums}
We need the following one-dimensional case of \cite[Theorem 3.1]{MST1}. The unweighted version of this result first appeared in \cite{W}, and \cite[Theorem 3.1]{MST1} can be deduced from it by summation by parts. Here is the set-up:

Let $P\in \RR[-]$ be a polynomial with real coefficients of degree $d \in\NN$ such that
\[	P(x) = P(x; c) = \sum_{j \leq d} c_j x^j.
\]
Suppose $I$ is an interval of length $\approx N$; we define
\[
	S_N = S_N(c) = \sum_{n \in I \cap \mathbb{Z}} e(P(n))\varphi(n). 
\] 
The function $\varphi:\mathbb{R} \rightarrow \mathbb{C}$ is assumed to be a $\mathcal
C^1\big(\RR\big)$ function which for some $C>0$ satisfies
\begin{align}
  \label{eq:217}
|\varphi(x)|\le C, \qquad \text{and} \qquad |\varphi'(x)|\le C(1+|x|)^{-1}.  \end{align}
Then: 
\begin{theorem}
	\label{thm:3}
        Assume that for some $k \leq d$
	\[
		\Big\lvert
		c_k - \frac{a}{q}
		\Big\rvert
		\leq
		\frac{1}{q^2}
	\]
        for some integers $a, q$ such that $0\le a\le q$ and $(a, q) =
        1$. Then for any $\alpha>0$ there is  $\beta_{\alpha}>0$ so that, for any
        $\beta\ge \beta_{\alpha}$, if
	\begin{equation}
		\label{eq:57}
		(\log N)^\beta \leq q \leq N^{k} (\log N)^{-\beta}
	\end{equation}
        then there is a constant $C>0$ 
	\begin{equation}
		\label{eq:56}
		|S_N|
		\leq
		C
		N (\log N)^{-\alpha}.
	\end{equation}
	The implied  constant $C$ is independent of $N$. 
\end{theorem}

Motivated by this Theorem, we define the following \emph{major boxes}:
\begin{definition}\label{logmaj}
For each $j \geq 1$, we define the $j$th major box, 
\[ \mathfrak{M}_j := \bigcup \mathfrak{M}_j(a_d,a_1;q),\]
where the union runs over all co-prime tuples, $(a_d,a_1,q)$ with $q \lesssim j^C$, and
\[ \mathfrak{M}_j(a_d,a_1;q) := \{ \xi = (\xi_d,\xi_1) \in \mathbb{T}^2 : \left\|\xi_i - \frac{a_i}{q} \right\|_{\mathbb{T}} \lesssim j^C 2^{-ji}, \ i = 1,d \},\]
where $\| x\|_{\mathbb{T}}$ denotes distance on the Torus, $\mathbb{R}/\mathbb{Z}$.
\end{definition}
We have the following claim:
\begin{lemma}\label{MAJ}
For any $\alpha > 0$ if the constant $C$ in the exponent in the definition of the major boxes is chosen sufficiently large, 
\[ \xi \notin \mathfrak{M}_j \Rightarrow |S_{2^j}(\xi)| \lesssim j^{-\alpha} 2^j.\]
\end{lemma}

\begin{proof}
We will prove that if $|S_{2^j}| \gtrsim j^{-\alpha}2^j$, then $\xi \in \mathfrak{M}_j$.

To do so, for each $i = 1,d$, we use Dirichlet's principle to choose a reduced $\frac{A_i}{Q_i}$ with
\[ Q_i \leq j^{-C/2} 2^{ji} \]
so that
\[ \left|\xi_i - \frac{A_i}{Q_i} \right| \leq \frac{j^{C/2}}{Q_i 2^{ji}} \leq \frac{1}{Q_i^2}. \]
If $C$ is chosen sufficiently large, then we are done unless each $Q_i \leq j^{C/2}$. So, assume contrary, set $q := \text{lcm}(Q_1,Q_d)$, and choose $\{ a_i\}$ so that
\[ \frac{A_i}{Q_i} = \frac{a}{q}, \ i = 1,d, \ (a_d, a_1;q) = 1.\]
Noting that $q \leq j^{C}$, we have shown that $\xi \in \mathfrak{M}_j(a_d,a_1;q)$, which yields the result.
\end{proof}

With this in hand, we turn to the proof of Theorem \ref{NoX_j}.
\subsection{The Argument}
\begin{proof}[Proof of Theorem \ref{NoX_j}]
Our goal is to establish \eqref{TT*goal}.

To this end, collect all the ``popular'' elements of $E(n)$ in
\[ P(n):=\{ x \in E(n) : x+h \in E(n) \text{ for some } 1 \leq h \lesssim j^C\}; \]
collect the complementary, ``lonely,'' elements of $E(n)$ in $L(n)$; by density considerations, $|L(n)| \ll j^{-C} 2^j$ for an appropriate choice of implicit constant. 
Then
\[ j^{-C} 2^j \lesssim |E(n)| \leq |L(n)| + |P(n)|,\]
so $|P(n)| \gtrsim j^{-C} 2^j$.

Our first claim is the following partial result.
\begin{lemma}\label{Y_j}
For $x \in P(n)$, $\mu(n), \ \mu(n)(n-x) \in Y_j := X_j + \{ |\beta| \lesssim j^C 2^{-(d-1)j} \}$.
\end{lemma}
\begin{proof}
Take some popular $x\in P(n)$, so that $x+h \in E(n)$ too, with $1 \leq h \lesssim j^C$. Assume as we may that both $|x|,|n| \lesssim 2^j$. Then there exists some rationals 
\[ \frac{A}{Q}, \ \frac{B}{Q}, \ Q \lesssim j^C, \ (A,B,Q) = 1 \]
depending on $x$, so that
\begin{equation}\label{termd}
\lambda(x) - \mu(n) \equiv \frac{A}{Q} + O( j^C 2^{-dj})
\end{equation}
and
\begin{equation}\label{termd-1}
dn \cdot \mu(n) - dx \cdot \lambda(x) \equiv \frac{B}{Q} + O( j^C 2^{-(d-1)j}),
\end{equation}
and similarly with $x$ replaced by $x+h$, and $\frac{A}{Q}, \frac{B}{Q}$ replaced appropriately as well. By considering 
\[ (dx) \times \eqref{termd} + \eqref{termd-1}\]
we deduce that
\begin{equation}\label{comb}
d(n-x) \mu(n) \equiv \frac{R}{Q} + O( j^C 2^{-(d-1)j} )
\end{equation}
for some integer $0 \leq R < Q$, which completes the second point. Applying the same reasoning with $x$ replaced by $x+h$ yields 
\begin{equation}\label{combh}
d(n-x-h) \mu(n) \equiv \frac{R'}{Q'} + O( j^C 2^{-(d-1)j} )
\end{equation}
for some other rational $\frac{R'}{Q'}$, $Q' \lesssim j^C$. Subtracting \eqref{combh} from \eqref{comb} and dividing by $dh \lesssim j^C$ completes the proof.
\end{proof}

Next, for each $\{ \frac{A}{Q} : Q \lesssim j^C\}$, set
\[ F_{\frac{A}{Q}} := \left\{ x \in P(n) : \mu(n)(n-x) \equiv \frac{A}{Q} + O(j^C 2^{-(d-1)j}) \right\}.\]
By the pigeon-hole principle, we know there exists some $\frac{A_0}{Q_0}$ such that
\[ |F_{\frac{A_0}{Q_0}}| \gtrsim j^{-C} 2^j.\]
Cover $F_{\frac{A_0}{Q_0}} \subset \bigcup J$, where each interval $J$ has length $j^{-C'} 2^j \lesssim |J| \ll j^{-C} 2^j$ for some $C' \gg C$; by another application of the pigeon-hole principle, there must be some $J$ such that
\[ |I| := |F_{\frac{A_0}{Q_0}} \cap J| \gtrsim j^{-C} 2^j.\]

Now, suppose that $v <u \in I$ are arbitrary. By definition, there are two integers, $l(v)$ and $l(u)$ so that
\begin{equation}\label{u}
\mu(n)(n-u) = l(u) + \frac{A_0}{Q_0} + O(j^C 2^{-(d-1)j})
\end{equation}
and
\begin{equation}\label{v}
\mu(n)(n-v) = l(v) + \frac{A_0}{Q_0} + O(j^C 2^{-(d-1)j}).
\end{equation}
Subtracting \eqref{u} from \eqref{v}, we see that
\begin{equation}\label{mu1}
\mu(n) = \frac{l(v) - l(u)}{u-v} + O(j^C 2^{-(d-1)j}).
\end{equation}
But, by Lemma \ref{Y_j}, we also have
\begin{equation}\label{mu2}
\mu(n) = \frac{A}{Q} + O(j^C 2^{-(d-1)j})
\end{equation}
for some $\frac{A}{Q}$ with $Q \lesssim j^C$;
by comparing denominators, if $\frac{l(v) - l(u)}{u-v} \neq \frac{A}{Q}$, then 
\[ u-v \gtrsim j^{-C} 2^{j(d-1)} \geq j^{-C} 2^j,\]
which is a contradiction since $|J| \ll j^{-C} 2^j$ is small. So, moving forward, we know that
\[ \frac{l(v) - l(u)}{u-v} = \frac{A}{Q} \]
for any $v<u$ in $I$, and thus
\[ l(u) = l(v) - \frac{A}{Q}(u-v).\]
Substituting this into \eqref{u}, we see that
\begin{equation}\label{almostmu}
\mu(n)(n-u) = l(v) - \frac{A}{Q}(u-v) + \frac{A_0}{Q_0} + O(j^C 2^{-(d-1)j}).
\end{equation}
Choose now $v < u$ to be two maximally spaced points in $I$, so that
\[ j^{-C'} 2^j \lesssim u-v \ll j^{-C} 2^j,\]
and subtract the previous identity \eqref{almostmu} from \eqref{v} above, to find
\[ \mu(n)(u-v) = \frac{A}{Q}(u-v) + O(j^C 2^{-(d-1)j});\]
dividing through by $u-v$ now shows that
\[ \mu(n) \in X_j,\]
as desired.
\end{proof}

\section{Most Weyl Sums are Safe: Another $TT^*$ Argument}\label{s:TT*2}
The goal of this section is to prove an $\ell^p$ estimate for certain maximal functions weighted by \emph{Weyl sums}, which we now proceed to introduce.

With $d \geq 2$ fixed in this section, define the Weyl sums,
\begin{equation}\label{Gauss}
S(A/Q,B/Q) := \frac{1}{Q} \sum_{r \leq Q} e(-A/Q \cdot r^d - B/Q \cdot r),
\end{equation}
where $(A,B,Q) = 1$. The fundamental estimate on these sums is due to Hua \cite[\S 7, Theorem 10.1]{Hua}.
\begin{proposition}\label{Hua}
For any $\epsilon > 0$, one may bound $|S(A/Q,B/Q)| \lesssim_\epsilon Q^{\epsilon - 1/d}$.
\end{proposition}
We also will need the following estimate on \emph{incomplete} Weyl sums, which may be deduced from the previous proposition by writing $\mathbf{1}_{[1,m]}$ as a weighted average of Dirichlet kernels.
\begin{lemma}\label{partial}
With $A,B,Q$ as above, for any $m \leq Q$
\[ \left| \sum_{r \leq m} e(-A/Q \cdot r^d - B/Q \cdot r) \right| \lesssim_\epsilon \min \left\{ m, Q^{1 - 1/d + \epsilon} \right\}.\]
\end{lemma}

Our final ingredient will be the following orthogonality property of Weyl sums, which we isolate in the below lemma.
\begin{lemma}\label{keyorth}
Suppose $(a,b,q) = 1$, but $(a,q) = v > 1$. Then $S(a/q,b/q) =0$.
\end{lemma}
\begin{proof}
Write $a/q = A/Q$ in reduced form, so that we have $b/q = b/Qv$, where $(b,v) = 1$ by assumption. Expand
\[ \aligned 
q \cdot S(a/q,b/q) &= \sum_{n \leq q} e( -A/Q \cdot n^d - b/Qv \cdot n)  \\
&\qquad = \sum_{s=0}^{v-1} \sum_{l=1}^Q e( - A/Q \cdot (sQ + l)^d - b/Qv \cdot (sQ + l)) \\
& \qquad \qquad = \sum_{s=0}^{v-1} \sum_{l=1}^Q e( -A/Q \cdot l^d - b/Qv \cdot l) e( - b/v \cdot s) \\
& \qquad \qquad \qquad = \sum_{l=1}^Q e( -A/Q \cdot l^d - b/Qv \cdot l) \times \sum_{s=0}^{v-1} e(-b/v \cdot s) \\
& \qquad \qquad \qquad \qquad = 0,
\endaligned \]
since $(b,v) = 1$, and $v > 1$. 
\end{proof}

Now, for each $s \geq 1$, collect, and any $1 \leq a \leq q \leq 2^s$, define the sets
\[ \mathcal{R}_s(a/q) := \{ b/q \text{ not necessarily reduced}: (a,b,q) = 1, \ 2^{s-1} \leq q < 2^s \};\]
note that we have the cardinality bound $|\mathcal{R}_s(a/q)| \lesssim 2^{2s}$, uniformly in $a/q$. We will also let $\varphi_s$ be a smooth (even) bump function supported in (say) a $2^{-5s}$ neighborhood of the origin. Let $\phi_s := \varphi_s^{\vee}$ denote its inverse Fourier transform.
Define now the maximal function
\begin{equation}\label{Weylmax}
M_sf := \sup_{1 \leq a \leq q \leq 2^s} \left| \sum_{\mathcal{R}_s(a/q)} S(a/q,b/q) \left( \varphi_s(\beta - b/q) \hat{f}(\beta) \right)^{\vee} \right|
\end{equation}
Using Lemma \ref{keyorth}, we prove the following estimate on $M_s$.
\begin{proposition}\label{Weylest}
For any $1 < p < \infty$, there exists an absolute $\eta = \eta(d,p) > 0$ so that we have the following norm bound, with implicit constant uniform in $s \geq 1$:
\[ \| M_sf \|_{\ell^p} \lesssim 2^{-\eta s} \|f\|_{\ell^p}.\]
Moreover, $\eta(d,2)$ can be taken to be $\frac{1}{2d^2} - \epsilon$ for any $\epsilon > 0$.
\end{proposition}

We begin by establishing $\ell^p$ estimates for $M_s$ without any decay; it will then suffice to establish Proposition \ref{Weylest} in the special case when $p = 2$.
\begin{lemma}
For any $1 \leq p \leq \infty$, $\| M_s f\|_{\ell^p} \lesssim \|f\|_{\ell^p}$.
\end{lemma}
\begin{proof}
It suffices to show that for any choice $a/q$,
\[ \left| \sum_{\mathcal{R}_s(a/q)} S(a/q,b/q) e(b/q \cdot x) \phi_s(x) \right| = |\phi_s|(x).\]
By Lemma \ref{keyorth}, we need only bound
\[ \left| \sum_{b \leq q} S(a/q,b/q) e(b/q \cdot x) \right| = 1 .\]
But the left-hand side of the foregoing can be re-expressed as
\[ \frac{1}{q} \sum_{r \leq q} e( -a/q \cdot r^d) \sum_{b \leq q} e( b/q \cdot (x-r)) =
\sum_{r \leq q, \ r \equiv x  \mod q} e( -a/q \cdot r^d),\]
from which the result follows.
\end{proof}

We now turn to the proof of Proposition \ref{Weylest}, which will follow a similar scheme to that of Theorem \ref{NoX_j}. Consider the following kernel:
\begin{align}
K_s(x,u) &:= \sum_{1 \leq b \leq r \leq 2^s} \sum_{1 \leq b' \leq r' \leq 2^s} R( a(x)/q(x), b/r) \cdot e(b/r \cdot x) \\
& \qquad \qquad \qquad \times I(x,u, b/r, b'/r') \cdot e( - b'/r' \cdot u) \cdot \overline{R(a'(u)/q'(u), b'/r')} \label{kerneldef}
\end{align}
where $a(x)/q(x)$ and $a'(u)/q'(u)$ are reduced rationals with denominators that are between $2^{s-1}$ and $2^s$, and we define
\[ I(x,u,b/r,b'/r')  := \sum_y \phi_s(x-y) \phi_s(y-u) e(-y \cdot (b/r - b'/r') )  \]
and
\[ R\left(a/q,b/r \right) := \frac{1}{\text{lcm}(q,r)} \sum_{n \leq \text{lcm}(q,r) } e\left( - a/q  \cdot n^d - b/r  \cdot n\right) \mathbf{1}_{2^{s-1} \leq 
\text{lcm}(q,r) < 2^s }. \]
The significance of this kernel is that, for an appropriate choice of $\frac{a(x)}{q(x)}, \ \frac{a'(u)}{q'(u)}$, we have
\[ \sum_u K_s(x,u) f(u) = T_s T_s^* f(x). \]
for $T_s$ a linearization of $M_s$. Accordingly, Proposition \ref{Weylest} will follow from the following key claim:

For any choice of $a(x)/q(x), \ a'(u)/q'(u)$, we may bound
\begin{equation}\label{TT*goal0}
|K_s(x,u)| \lesssim_\epsilon  2^{(\epsilon - 1/d^2)s} |\phi_s*\phi_s(x-u)|.
\end{equation}

With this goal in mind, we proceed to the proof.
\begin{proof}[Proof of Proposition \ref{Weylest}]
Our task is to establish \eqref{TT*goal0}.

For notational ease, abbreviate
\[ a/q = {a(x)}/{q(x)}, \ a'/q' = a'(u)/q'(u).\]
By Lemma \ref{keyorth}, we know that $R(a/q,b/r)$ vanishes unless $2^{s-1} \leq q < 2^s$ and $r$ divides $q$. By Poisson summation, we also see that $I(x,u,b/r,b'/r') = 0$ unless $b/r = b'/r'$. Indeed, note that
\[ 2^{-2s} \leq \frac{1}{rr'} \leq \|b/r - b'/r'\|_{\mathbb{T}} \]
for $b/r \neq b'/r'$ with $r,r' \leq 2^s$, so that
\[ \varphi_s(\xi + m  + (b/r- b'/r') ) \varphi_s( \xi) \equiv 0 \]
for all $m \in \mathbb{Z}$ (recall $\varphi_s$ is supported in $\{ |\xi| \lesssim 2^{-5s} \}$).
Consequently, 
we see that to establish \eqref{TT*goal}, we may replace the left-hand side with
\begin{equation}\label{TT*red}
\sum_{\theta \in \Z/ Q \Z} R(a/q,\theta) \cdot e(\theta x) \cdot \phi_s*\phi_s(x-u) \cdot e(-\theta u) \cdot \overline{R(a'/q',\theta)},
\end{equation}
where we set
\begin{equation}
Q = \text{gcd}(q,q');
\end{equation}
for future reference we also set 
\[ p := \frac{q}{Q}, \ p' := \frac{q'}{Q}.\]
We now expand out every sum, and use the orthogonality relationship
\[
\sum_{\theta \in \Z / Q \Z} e(\theta \cdot x) = Q \cdot \mathbf{1}_{x \equiv 0 \mod Q}(x).\]
We get
\[ \frac{1}{qq'} \sum_{r \leq q, \ s \leq q'} e(-a/q \cdot r^d + a'/q' \cdot s^d) \left( \sum_{\theta \in \mathbb{Z}/Q \mathbb{Z}} e((r +x - s - u) \cdot \theta ) \right) \times \phi_s * \phi_s(x-u),\]
so we see that we need bound
\begin{equation}\label{TT*red2}
\frac{1}{q q'} \cdot Q \cdot \sum_{s \leq q'} \sum_{i \leq p} e\left( - a/q \cdot \Big(s + (u-x) + (j_0 +i)Q \Big)^d + a'/q' \cdot s^d \right) = O(2^{(\epsilon - 1/d^2)s}) 
\end{equation}
uniformly in $x,u$. Here, $j_0$ is an integer depending only on $q,q',x,u$.

There are two estimates now available for the double sum: summing in $s \leq q'$ and using Lemma \ref{partial}, we may bound the double sum:
\begin{equation}\label{min1}
\lesssim_\epsilon p \cdot \left( \frac{qq'}{Q} \right)^{1 - \frac{1}{d} + \epsilon}.
\end{equation}
Alternatively, if we sum in $i$ first, and use Hua's Proposition \ref{Hua}, we may bound the double sum by
\begin{equation}\label{min2}
\lesssim_\epsilon q' \cdot p^{1-\frac{1}{d} + \epsilon}.
\end{equation}
If $p \lesssim 2^{s/d}$, the estimate \eqref{min1} leads to the desired bound; otherwise \eqref{min2} is effective: in particular, we are left with an upper estimate for the left hand side of \eqref{TT*red2} of
\[ 
\lesssim_\epsilon \min\{ p^{1-1/d + \epsilon} \cdot (q')^{\epsilon - 1/d}, p^{\epsilon - 1/d} \} \leq 
\min\{ p^{1-1/d + \epsilon} \cdot 2^{s(\epsilon - 1/d)}, p^{\epsilon - 1/d} \} \lesssim 2^{(\epsilon - 1/d^2) s}, \]
which yields the result.
\end{proof}

\section{Approximations}\label{s:app}
In this section, we construct analytic approximates to the ``single scale'' multipliers 
\begin{equation}\label{Mj}
M_j(\lambda,\beta) := \sum_m \psi_j(m) e(-\lambda m^d - \beta m).
\end{equation}
Throughout this section, $0 < \epsilon \ll 1$ will denote a sufficiently small constant (which may depend on $d,p$).

With $\chi$ defined as in \eqref{chi}, for each $j \geq 1$, let
\[ \Xi_j(t) := \mathbf{1}_{|t| \leq 2^{-c(j)}}(t), \]
where $c(j)$ is defined in \eqref{c(j)} (so in particular $2^{c(j)} \approx j^{-C_{d,p}} 2^{dj}$).

The rationals in the two torus are the union over $ s\in \mathbb N $ of the collections 
\begin{equation}\label{e:Rs}
\R_s := \{ (A/Q,B/Q) \in \TT^2 : (A,B,Q) = 1, 2^{s-1 }\leq Q < 2^{s} \}.
\end{equation}
For each $s , j \in \mathbb N $, define the multiplier
\begin{align}\label{e:Ljs}
 L_{j,s}(\lambda,\beta)& := \sum_{ (A/Q,B/Q)\in  \R_s} S(A/Q,B/Q) H_j(\lambda - A/Q, \beta - B/Q) \Xi_j(\lambda - A/Q) 
\\ & \qquad \qquad \qquad  \times \chi_{s}(\lambda- A/Q) \chi_s(\beta - B/Q), 
 \\ 
  \noalign{\noindent where  $\chi_s(t) := \chi(2^{2^{s d \kappa }} t)$ where $\kappa = \kappa(p)$ is a sufficiently small number;  a continuous analogue of the sum is given by }   \label{e:Hj}
   &H_j(x,y) := \int e(-x t^d - y t) \psi_j(t) \; dt  ; 
   \\ \noalign{\noindent and we recall the complete Gauss sum is given by }
    \label{e:gauss} 
    &S(A/Q,B/Q) := \frac{1}{Q} \sum_{r=0}^{Q-1} e(-A/Q \cdot r^d - B/Q \cdot r). 
\end{align}

Before proceeding, we will use the orthogonality relationship of Gauss sums to re-express $L_{j,s}(\lambda,\beta)$ in the following convenient form:
\begin{align}\label{Ljs}
 L_{j,s}(\lambda,\beta)& := 
 \sum_{2^{s-1} \leq Q < 2^s, \ (A,Q) =1} \sum_{B\leq Q} S(A/Q,B/Q) H_j(\lambda - A/Q,\beta-B/Q) \Xi_j(\lambda - A/Q) \\
& \qquad \qquad \qquad \times \chi_{s}(\lambda- A/Q) \chi_s(\beta - B/Q). 
\end{align}
We now employ the beautiful multiplier theory of \cite{MST1}. Specifically, for each $s \geq 1$, let
$\mathcal{U}_{2^s}$
be as in Theorem \ref{Multtheorem}, where $0< \rho \ll 1$ is a sufficiently small constant (depending on all other parameters; one may think of $\rho = \kappa^{C_{p,d}}$). Define now the following two multipliers, the composition of which is $L_{j,s}$ for $s \gg_\rho 1$ sufficiently large:
\begin{equation}\label{Lj1}
L_{j,s}^1(\lambda,\beta) := \sum_{2^{s-1} \leq Q < 2^s, \ (A,Q) =1} \sum_{\theta \in \mathcal{U}_{2^s}} H_j(\lambda - A/Q,\beta-\theta) \Xi_j(\lambda - A/Q) \chi_s(\lambda - A/Q) \chi_s(\beta - \theta)
\end{equation}
and
\begin{equation}\label{Ls2}
L_{s}^2(\lambda,\beta) := \sum_{2^{s-1} \leq Q < 2^s, \ (A,Q) =1} \sum_{B\leq Q} S(A/Q,B/Q) \mathbf{1}_{|\lambda - A/Q| \lesssim 2^{-2^{sd \kappa}}}(\lambda- A/Q) \overline{\chi}_s(\beta - B/Q).
\end{equation}
Here, $\overline{\chi}_s$ is defined as is $\chi_s$.
By Proposition \ref{Weylest} above, we know that 
\[ \sup_\lambda \left| \left( L_{s}^2(\lambda,\beta) \hat{f}(\beta) \right)^{\vee} \right|\]
has $\ell^p$ operator norm decaying exponentially in $s$. 

We will eventually need to estimate
\[ \sup_{\lambda} \left| \left( \sum_{j : j^C \gtrsim 2^s} L_{j,s}^1(\lambda,\beta) \hat{f}(\beta) \right)^{\vee} \right|; \]
first, though, to dispose of certain ``error terms,'' we need to estimate
\begin{equation}\label{LforE}
\sum_{s : 2^s \lesssim j^C} \sup_{\lambda} \left| \left( L_{j,s}(\lambda,\beta) \hat{f}(\beta) \right)^{\vee} \right|.
\end{equation}
Our tool will be the Sobolev embedding Lemma \ref{sobemb}. We state the relevant estimates in the form of the following lemma.
\begin{lemma}\label{triangleest}
Suppose $2^s \lesssim j^C$. Then, in the language of Lemma \ref{sobemb}, applied to the multipliers $L_{j,s}$, one may take
\[ a(p) \lesssim_\epsilon 2^{(\epsilon - 1/d)s}, \ A(p) \lesssim_\epsilon 2^{(\epsilon - 1/d)s} \cdot 2^{dj} \]
for any $1 < p < \infty$. Consequently,
\[
\| \sup_{\lambda} \left| \left( L_{j,s}^1(\lambda,\beta) \hat{f}(\beta) \right)^{\vee} \right| \|_{\ell^p} \lesssim_p j^C \|f\|_{\ell^p},
\]
and
\[ \| \eqref{LforE} \|_{\ell^p} \lesssim j^C \|f\|_{\ell^p} \]
as well.
\end{lemma}
\begin{proof}
Since $H_j(\lambda,\beta)$ is an $L^p$ multiplier uniformly in $\lambda$ (it's inverse Fourier transform is dominated by the continuous Hardy-Littlewood maximal function), by Theorem \ref{Multtheorem}, we obtain the desired estimate on $a(p)$; here we used that the cut-offs in $\lambda$ are disjointly supported. On the interior of $X_j$, it's $\lambda$- derivative analogously is a linear combination of two terms, the first of which, in light of Theorem \ref{Multtheorem}, has $\ell^p$ norm bounded by 
\[ \lesssim_\epsilon s \cdot 2^{(\epsilon - 1/d)s} \cdot 2^{dj},\] and the second of which has $\ell^p$ norm bounded by 
\[ \lesssim_\epsilon s \cdot 2^{(\epsilon - 1/d)s} \cdot 2^{2^{s \kappa d}} \lesssim_\epsilon 2^{(\epsilon - 1/d)s} \cdot 2^{dj}\] 
provided $2^s \lesssim j^C$, since we have chosen $\kappa$ sufficiently small. The rest follows from Lemma \ref{sobemb}; note our use of Proposition \ref{Hua}.
\end{proof}

We now define
\begin{equation}  \label{e:Lj}
L _{j} (\lambda , \beta ) = \sum_{s : 2^s \lesssim j^{C } } L_{j,s} (\lambda , \beta ),
\end{equation}
where $C = C_{d,p}$ is the same constant appearing in the definition of $X_j$, see \eqref{X_j},
and decompose that 
\begin{equation}\label{split}
M_j(\lambda, \beta) \mathbf{1}_{X_j}(\lambda) = L_j(\lambda,\beta) + \mathcal{E}_j(\lambda,\beta),
\end{equation}
where we have
\begin{equation} \label{e:Mjdef} 
 M_j(\lambda,\beta) := \sum_{m} e(- \lambda m^d - \beta m ) \psi_j(m),  
\end{equation}
the $j$th block of the multiplier. We make two remarks: first, with this definition, it is clear that $\mathcal{E}_j(\cdot, \beta)$ is supported in $X_j$; second the multiplier $M_j$ is bounded on $\ell^p$, independent of $\lambda$, since its inverse Fourier transform is trivially bounded by $M_{HL}$.

By the triangle inequality, and the behavior of $\left(M_j(\lambda,\cdot) \hat{f} \right)^{\vee}$ in light of Lemma \ref{triangleest}, we have the following estimates on the $\ell^p$ operator norms of the Fourier multipliers $\mathcal{E}_j(\lambda,\beta)$:
\begin{lemma}\label{Lperrors}
For any $1 < p <\infty$, uniformly in $\lambda$ we have:
\begin{equation}\label{Lpnoderiv}
\| \left( \mathcal{E}_j(\lambda,\beta) \hat{f}(\beta) \right)^{\vee} \|_{\ell^p} \lesssim j^C \cdot \|f \|_{\ell^p}; \end{equation}
for $\lambda$ in the interior of $X_j$ we also have
\begin{equation}\label{Lpderiv}
\| \left( \partial_\lambda \mathcal{E}_j(\lambda,\beta) \hat{f}(\beta) \right)^{\vee} \|_{\ell^p} \lesssim j^C \cdot 2^{dj} \cdot \|f \|_{\ell^p}. 
\end{equation}
\end{lemma}

We now refine this estimate in $\ell^2$, in the following Proposition.

\begin{proposition}\label{t:smooth}
Uniformly in $(\lambda,\beta) \in \mathbb{T}^2$, $\lambda \in X_j$, the following estimates hold:
\begin{equation}\label{noderiv}
| \mathcal{E}_j(\lambda,\beta) | \lesssim j^{-\frac{1}{2\kappa}},
\end{equation}
and for $\lambda$ in the interior of $X_j$, for any $\beta \in \mathbb{T}$ we have
\begin{equation}\label{deriv}
| \partial_\lambda \mathcal{E}_j(\lambda,\beta) | \lesssim 2^{dj}.
\end{equation}
\end{proposition}

Interpolating between Lemma \ref{Lperrors} and Proposition \ref{t:smooth}, yields the following Proposition. 
\begin{proposition}\label{errorprop} Let $1 < p < \infty$, and suppose $\kappa = \kappa(p)$ has been chosen sufficiently small. Then the error term
\begin{equation}\label{killerrors} 
\sum_{j \geq 1} \sup_{\lambda} \left| \left( \mathcal{E}_j(\lambda, \beta) \hat{f}(\beta) \right)^{\vee} \right|
\end{equation}
is bounded on $\ell^p$.
\end{proposition}
\begin{proof}
In the language of Lemma~\ref{sobemb}, we have the following estimates on $\left( \mathcal{E}_j(\lambda, \beta) \hat{f}(\beta) \right)^{\vee}$:
\[ a(p) \lesssim j^{-c_p/\kappa}, \ A(p) \lesssim  j^C \cdot 2^{dj}.\]
The result follows from Lemma \ref{sobemb} and the triangle inequality.
\end{proof}
We turn to the proof of Proposition \ref{t:smooth}, which is a standard application of the Hardy-Littlewood method in exponential sums. 
A core definition in this method is that of \emph{major boxes}, which we define slightly differently here than in Definition \ref{logmaj} above. 

%%%%%%%%%%  Definition 
\begin{definition}  \label{d:major}
For $ (A/Q,B/Q)\in  \mathcal R_s$, where $ s\leq j \epsilon $, define the \emph{$ j$th major box at $ (A/Q, B/Q)$} to be 
the rectangle in $ \mathbb T ^2 $ given by 
\begin{equation}\label{e:major}
\mathfrak{M}_j(A/Q,B/Q) := \left\{ (\lambda, \beta) \in \mathbb{T}^2 : \| \lambda - A/Q\|_{\mathbb{T}} \leq 2^{(\epsilon - d)j},
\| \beta - B/Q\|_{\mathbb{T}} \leq 2^{(\epsilon - 1)j} \right\}.
\end{equation}
\end{definition}

We collect the major boxes 
\begin{equation}\label{e:Major}
 \mathfrak{M}_j := \bigcup_{(A,B,Q) = 1 : Q \leq 2^{\epsilon j} } \mathfrak{M}_j(A/Q,B/Q). 
\end{equation}

The union above is over  disjoint sets:  if 
\[ (\lambda,\beta) \in \mathfrak{M}_j(A/Q,B/Q) \cap \mathfrak{M}_j(A'/Q',B'/Q'),\]
and 
$ A/Q \neq A'/Q'$, then 
\[ 2 \cdot 2^{(\epsilon-d)j} \geq |A/Q - \lambda| + |A'/Q' - \lambda| \geq |A/Q - A'/Q'| \geq \frac{1}{2^{2 \epsilon j}},\]
which is a contradiction for $\epsilon > 0$ sufficiently small.   If $ A/Q=A'/Q'$, then necessarily
$B/Q \neq B'/Q'$ and the same argument applies.

On any fixed major box, we have this approximation of $ M _{j} (\lambda , \beta )$, which is at the 
core of the proof of Theorem~\ref{t:smooth}.

%%%%%%%%%%%%%%%%%%%%%%%%%%%%%% Lemma 
\begin{lemma}\label{MAJor}  For $ 1\leq s \leq \epsilon j$,  $(A/Q,B/Q)\in  \mathcal R_s$,  and 
$ (\lambda , \beta ) \in \mathfrak M _{j} (A/Q, B/Q)$,  we have the approximation 
\begin{equation}\label{e:smooth}
 M_j(\lambda, \beta) = S(A/Q,B/Q) H_j(\lambda - A/Q, \beta - B/Q) + O(2^{(2\epsilon -1 )j }).
\end{equation}
The terms  above are defined in \eqref{e:Mjdef}, \eqref{e:gauss}, and \eqref{e:Hj}, respectively.  

In particular, if $(A/Q,B/Q) \in \mathcal{R}_s$, and $(A,Q) > 1$, then $M_j(\lambda,\beta) = O(2^{(2\epsilon -1)j})$.
\end{lemma}
This argument essentially appeared as \cite[Lemma 4.14]{KL}, in the case where $d=2$. The proof in this case is entirely analogous.
%%%%%%%%%%%%%%%%%%%%%%%%%%%%%% PROOF PROOF PROOF
\begin{proof}
Throughout the proof we write 
\begin{equation}\label{e:etas}
 \lambda = A/Q + \eta_d, \ \beta = B/Q + \eta_1,
\end{equation}
where $|\eta_d| \leq 2^{(\epsilon-d)j}$, and $|\eta_1|\leq 2^{(\epsilon-1)j}$.

The sum $ M _{j} (\lambda , \beta )$  is over integers, positive and negative,  in the support of $ \psi _j $.  
We  consider the sum over positive $m$, and decompose  into residue classes $ \operatorname {mod} Q$. 
Thus write $ m = pQ+r$, where $0 \leq r < Q \leq 2^{j\epsilon}$, 
and the integers $ p$ take values in an interval $[c,d] = [c_j(Q), d_j(Q)]$, 
 in order to cover the support of $ \psi _j$. 

The argument of the exponential in \eqref{e:Mjdef} is, after reductions modulo 1,  
\[ \aligned
\lambda m^d + \beta m &= (A/Q + \eta_d)(p Q + r )^d - (B/Q + \eta_1)(pQ + r) \\
&\equiv r^d \cdot A/Q + r \cdot B/Q + (pQ)^d \cdot \eta_d + pQ \cdot  \eta_1 + O(2^{(2\epsilon-1)j}). \endaligned\]
That is, we can write 
\begin{equation*}
e (- \lambda m ^d - \beta m) = 
e ( - r^d A/Q - r \cdot B/Q - (pQ)^d \cdot  \eta_d - pQ \cdot \eta_1 ) + O(2^{(2\epsilon-1)j}).  
\end{equation*}

Then, we can write the sum $ \sum _{m \geq 0}
e(-\lambda m^d - \beta m) \psi_j(m)$ as follows.
\[ \aligned 
 &
\sum_{p \in I} \sum_{r=0}^{Q-1} 
e\big( - (r^d \cdot A/Q + r \cdot B/Q + (pQ)^d \cdot  \eta_d + pQ  \cdot \eta_1 ) \big) \psi_j(pQ + r)  + O(2^{(2\epsilon -1)j}) \\
& \qquad =   \sum_{r=0}^{Q-1}  e(-r^d \cdot A/Q - r \cdot B/Q)  \times \sum_{p \in I} 
e( - \eta_d  \cdot (pQ)^d - \eta_1 \cdot  pQ)\psi_j(pQ)  + O(2^{(2\epsilon -1)j}),
\\  \label{e:Ssum}
& \qquad \qquad = S (A/Q, B/Q) \times Q\cdot \sum_{p \in I} 
e( - \eta_d  \cdot (pQ)^d - \eta_1 \cdot pQ)\psi_j(pQ)    + O(2^{(2\epsilon -1)j}). 
\endaligned \]
Above, we have appealed to several elementary steps. 
One of these is that $ \sum_{j} \lvert  \psi _j (m)\rvert \lesssim 1 $. 
Some additional terms in $r$ have been added,  so that the sum over $ p$ and $ r$ are over independent sets. 
These additions are absorbed into the Big-$O $ term.  
The argument of $ \psi _j $ is changed from $ pQ+r$ to $ pQ$, in view of the fact that 
the derivative of $ \psi _j$ is at most $ 2 ^{-2j}$, with the  change also being  absorbed into the Big-$O$ term.  
Finally, we appeal to the definition of the Weyl sum in \eqref{e:gauss} in order 
to have $ S (A/Q, B/Q) $ appear in the last line.  

Comparing \eqref{e:Ssum} to the desired conclusion \eqref{e:smooth}, we show that 
\begin{equation}\label{e:SSum}
Q \cdot \sum_{p \in I}  e( - \eta_d  \cdot (pQ)^d - \eta_1  \cdot pQ)  \psi_j(pQ)   = \int_{0}^\infty e( - \eta_d  t^d  - \eta_1 t) \psi_j(t) \; dt + O(2^{(2\epsilon - 1)j}). 
\end{equation}
The same argument to this point will apply to the sum over negative $ m$, so that our proof will then be complete.

But the proof of \eqref{e:SSum} is straight forward.  For fixed $ p\in I$, and $ 0\leq t \leq Q$, we have 
\begin{align*}
\bigl\lvert 
 e( -\eta_d  \cdot (pQ)^d  &- \eta_1  \cdot pQ)   \psi_j(pQ) -  e( - \eta_d  \cdot (pQ+t)^2 - \eta_1  \cdot (pQ+t))  \psi_j(pQ+t) 
\bigr\rvert
\\& \lesssim  
\lvert    e (-\eta_d  \cdot (pQ)^d) - e (-\eta_d  \cdot (pQ+t)^d)\rvert 2 ^{-j} 
\\
& \qquad + \lvert    e (-\eta _1  \cdot pQ ) - e (-\eta _1  \cdot (pQ+t))\rvert  2 ^{-j} 
\\& \qquad \qquad +  
\lvert  \psi_j(pQ)  -  \psi_j(pQ+t)  \rvert.  
\end{align*}
Each of the three terms on the right is at most $ O (2 ^{(2 \epsilon -2)j})$. 
In view of the fact that there are $ Q \lvert  I\rvert \lesssim 2 ^{j} $ summands on the left in 
\eqref{e:SSum}, this is  all that we need to conclude the inequality in \eqref{e:SSum}. 
The three terms on the right above are bounded in reverse order. 
    Since the derivative of $ \psi _j$ is at most $ 2 ^{-2j}$, 
\begin{equation*}
\lvert  \psi_j(pQ)  -  \psi_j(pQ+t)  \rvert \lesssim 2 ^{-2j}   t  \lesssim 2 ^{ (\epsilon -2)j},  
\end{equation*}
since $0\leq  t\leq Q \leq 2 ^{\epsilon j}$.  Recalling that $ \lvert  \eta _1\rvert \leq 2 ^{ (\epsilon -1)j} $, there holds 
\begin{equation*}
2^{-j} \cdot \lvert    e (- \eta _1 \cdot  pQ ) - e (-\eta _1  \cdot (pQ+t))\rvert  \leq 
 t 2^{-j} \cdot \lvert  \eta _1 \rvert \leq   2 ^{ (2 \epsilon -2)j}.  
\end{equation*}
Recalling that $ \lvert  \eta _d\rvert\leq 2 ^{ (\epsilon -d)j} $, there holds 
\begin{equation*}
2^{-j} \lvert    e (\eta_d  \cdot (pQ)^d) - e (\eta_d \cdot  (pQ+t)^d)\rvert \lesssim 
t (p Q)^{d-1} \cdot 2^{-j} \lvert  \eta _d \rvert  \lesssim 2 ^{ (2 \epsilon -2j)}.  
\end{equation*}
Thus, \eqref{e:SSum} holds.  

\end{proof}
%%%%%%%%%%%%%%%%%%%%%%%%%%%%%% PROOF PROOF PROOF

This Lemma is motivated by \cite[Lemma 4.6]{B1}.  

%%%%%%%%%%%%%%%%%%%%%%%%%%%%%% LEMMA LEMMA LEMMA
\begin{lemma}\label{l:Lj<}  For $ \epsilon  > 0$ sufficiently small, there exists an $\eta = \eta(\epsilon) > 0$ so that the following estimates are satisfied for all integers $j \geq 1$:

First, for $s : 2^s \lesssim j^C$, there holds 
\begin{equation}\label{e:LJ<}  
\sum_{ 1 \leq s' \neq s : 2^s, 2^{s'} \lesssim j^C}\lvert   L _{j,s'} (\lambda , \beta )\rvert \lesssim j^C 2^{-j/2} \qquad 
(\lambda , \beta ) \in \bigcup _{ (A/Q,B/Q)\in  \mathcal R_s}  \mathfrak M _{j} (A,B,Q).  
\end{equation}
For any $(\lambda,\beta) \in \mathfrak{M}_j(A,B,Q)$ with $(A/Q,B/Q)\in  \mathcal R_s$ for $2^s \lesssim j^C$
\begin{equation}\label{e:LJ=} 
|M_j(\lambda,\beta) - L_{j,s}(\lambda,\beta)| \lesssim j^{-\frac{1}{2\kappa}}.
\end{equation}
For any $(\lambda,\beta) \in \mathfrak{M}_j(A,B,Q)$ with $(A/Q,B/Q)\in  \mathcal R_s$ for $j^C \ll 2^s \lesssim 2^{\epsilon j}$
\begin{equation}\label{e:MJ=} 
|M_j(\lambda,\beta) \mathbf{1}_{X_j}(\lambda)| \lesssim 2^{(2\epsilon-1)j},
\end{equation}
and
\begin{equation}\label{e:LJoff=} 
|L_j(\lambda,\beta)| \lesssim 2^{- \epsilon j /2 }.
\end{equation}
Finally,
\begin{equation}\label{e:Lj>}  
|M_j(\lambda,\beta)| \lesssim 2^{-\eta j}, \ \lvert   L _{j} (\lambda , \beta )\rvert \lesssim 2 ^{ -  \epsilon j/2} \qquad 
(\lambda , \beta ) \not\in  \mathfrak M _{j} . 
\end{equation}
\end{lemma}
Combining the points of the Lemma, we quickly prove Proposition \ref{t:smooth}.
\begin{proof}[Proof of Proposition \ref{t:smooth}, Assuming Lemma \ref{l:Lj<}]
Off the major boxes centered at rational points with denominators $\lesssim j^C$, the result is clear, since both terms have a power savings. Consider the converse case, where the denominator has magnitude $\approx 2^{s_0}$ for some $2^{s_0} \lesssim j^{C}$.

We express
\begin{equation}\label{sep}
\mathcal{E}_j(\lambda,\beta) = M_j(\lambda,\beta) - L_j(\lambda,\beta) = \left( M_j(\lambda,\beta) - L_{j,s_0}(\lambda,\beta) \right) + \sum_{s : 2^s \lesssim j^C, s \neq s_0} L_{j,s}(\lambda,\beta)
\end{equation}
and consolidate, to obtain \eqref{noderiv}; \eqref{deriv} is just a trivial derivative estimate.
\end{proof}
%%%%%%%%%%%%%%%%%%%%%%%%%%%%%% LEMMA LEMMA LEMMA

%%%%%%%%%%%%%%%%%%%%%%%%%%%%%% PROOF PROOF PROOF
\begin{proof}[Proof of Lemma \ref{l:Lj<}]
One first observes that 
\[ \{ \chi_s(\lambda - A/Q) \chi_s(\beta - B/Q) : (A/Q,B/Q) \in \mathcal{R}_s \}\]
 are disjointly supported in $(\lambda,\beta)$.

First suppose that $(\lambda,\beta) \in \mathfrak{M}_j(A/Q,B/Q)$ for some $2^{s_0-1} \leq Q < 2^{s_0} \lesssim j^C$. The first observation is that for any $(A'/Q',B'/Q') \in \mathcal{R}_s$,
\[ 2^{-2s} \lesssim |A/Q - A'/Q'| + |B/Q - B'/Q'|,\]
and thus
\[ 2^{-2s} \lesssim |\lambda - A'/Q'| + |\beta - B'/Q'|\]
as well. Consequently, we have the estimate
\[ | H_j(\lambda - A'/Q',\beta - B'/Q') | \lesssim (2^j 2^{-2s} )^{-1/2} \lesssim j^C 2^{-j/2},\]
so
\[ \sum_{s : 2^s \lesssim j^C, s \neq s_0} L_{j,s}(\lambda,\beta) = O(j^C 2^{-j/2}) \]
on $\mathfrak{M}_j(A/Q,B/Q)$. 

To establish \eqref{e:LJ=}, up to an error of $O(2^{(2\epsilon -1)j})$, it suffices to compare
\begin{equation}\label{MAJBOX}
S(A/Q,B/Q) H_j(\lambda - A/Q,\beta - B/Q) \cdot \left( 1 - \chi_s(\lambda - A/Q) \chi_s(\beta - B/Q) \right).
\end{equation}
If $j \gg_d 2^{\kappa s/d}$, then \eqref{MAJBOX} vanishes identically. Otherwise,
\[ 2^s \gtrsim j^{d/\kappa},\]
and the decay of the Weyl sum allows one to estimate 
\[ \eqref{MAJBOX} = O(2^{(\epsilon - 1/d) s}) = O( 2^{-\frac{s}{2d}} ) = O( j^{- \frac{1}{2 \kappa} }).\]
Next, assume that $\lambda \in X_j$, and that for some $\beta$, $(\lambda,\beta) \in \mathfrak{M}_j(A/Q,B/Q)$ for some $(A/Q,B/Q) \in \mathcal{R}_s$ with $j^C \ll 2^s \lesssim 2^{\epsilon j}$. If $(A,Q) > 1$, then we are done, so -- seeking a contradiction -- we may assume that $(A,Q) = 1$ is reduced, so that we have
\begin{equation}\label{linMaj}
|\lambda - A/Q| \lesssim 2^{(\epsilon - d)j} 
\end{equation}
for some $j^C \ll Q \lesssim 2^{\epsilon j}$.
But, $\lambda \in X_j$, which means that there is some reduced rational $a/q$, with $q \lesssim j^C$, so that
\begin{equation}\label{linX_j}
|\lambda - a/q| \lesssim j^C 2^{-dj}.
\end{equation}
Taking into account \eqref{linMaj} and \eqref{linX_j}, we are left with the following chain of inequalities:
\[ \frac{1}{j^C 2^{\epsilon j}} \lesssim |a/q - A/Q| \lesssim 2^{(\epsilon - d)j} \]
which yields our desired contradiction.

Next, suppose that 
\[ (\lambda,\beta) \notin \bigcup_{Q \lesssim j^C} \mathfrak{M}_j(A/Q,B/Q),\]
%but $(\lambda,\beta) \in \mathfrak{M}_j(A/Q,B/Q)$
so that whenever
\[ |\lambda - A/Q| \lesssim 2^{(\epsilon - d)j}, \ |\beta - B/Q| \lesssim 2^{(\epsilon -1)j},\]
necessarily we have $Q \gg j^C$. So,
\begin{equation}\label{IBP}
|H_j(\lambda- A/Q,\beta - B/Q) | \lesssim 2^{-\epsilon j / 2 }
\end{equation}
for any $(A/Q,B/Q) \in \mathcal{R}_s, \ 2^s \lesssim j^C$, so
\[ |L_j(\lambda,\beta) | \lesssim 2^{-\epsilon j /2}, \]
upon taking into account the geometric decay of the Gauss sums. The key estimate we used in establishing \eqref{IBP} is the stationary phase estimate,
\[ |H_j(x,y)| \lesssim (1+ 2^{dj}|x| + 2^j|y|)^{-1/2},\]
which follows by standard arguments (see, for instance, the proof of \cite[Lemma 4.18]{KL}). This same argument applies to the case where $(\lambda,\beta) \notin \mathfrak{M}_j$.

As for $M_j(\lambda,\beta)$, by Dirichlet's principle, we may choose two reduced rationals \[ a/q, b/r, \text{ with } q \lesssim 2^{(d-\epsilon)j}, \ r \lesssim 2^{(1-\epsilon)j},\] so that
\[ |\lambda - a/q| \lesssim \frac{1}{q 2^{(d-\epsilon)j} }, \ 
|\beta - b/r| \lesssim \frac{1}{r 2^{(1-\epsilon)j} }.\]
By \cite[Corollary, p. 1304]{SW0}, we know that $M_j(\lambda,\beta) = O(2^{-\eta j})$ unless both $q, \ r \lesssim 2^{\eta' j}$ for some $\eta$ which goes to zero with $\eta'$. But then $Q:=\text{lcm}(q,r) \lesssim 2^{2 \eta' j}$; setting $A/Q = a/q, B/Q = b/r$ exhibits $(\lambda,\beta) \in \mathfrak{M}_j(A/Q,B/Q)$ if $\eta'$ is sufficiently small. This contradiction shows that we indeed have $M_j(\lambda,\beta) = O(2^{-\eta j})$, as desired.
\end{proof}

%%%%%%%%%%%%%%%%%%%%%%%%%%%%%% PROOF PROOF PROOF

\section{Completing the Proof, $p \geq 2$}\label{s:comp}
With the approximations from the previous section in mind, we decompose our maximal operator
\[ \aligned 
& \mathcal{C}_d f  \leq
\sup_\lambda \left| \sum_{j \geq 1} \mathbf{1}_{X_j}(\lambda) \left(M_j(\lambda,\beta) \hat{f}(\beta) \right)^{\vee} \right| + 
\sum_{j \geq 1} \sup_{\lambda \notin X_j} \left| \left(M_j(\lambda,\beta) \hat{f}(\beta) \right)^{\vee} \right|
\\
& \qquad \leq\sup_\lambda \left| \sum_{j \geq 1} \left(L_j(\lambda,\beta) \hat{f}(\beta) \right)^{\vee} \right| + \sum_{j \geq 1} \sup_{\lambda \notin X_j} \left| \left(M_j(\lambda,\beta) \hat{f}(\beta) \right)^{\vee} \right|
+ \sum_{j \geq 1} \sup_\lambda \left| \left( \mathcal{E}_j(\lambda,\beta) \hat{f}(\beta) \right)^{\vee} \right|, \endaligned\]
where the second and the third term have bounded $\ell^p$ norm by Theorem \ref{NoX_j} and Proposition \ref{errorprop}. Here, $M_j, L_j$ and $X_j$ are defined in \eqref{Mj}, \eqref{e:Lj}, and \eqref{X_j} respectively.

We now reverse the order of summation to consider the following family of maximal functions, indexed by $s \geq 1$:
\begin{equation}\label{max-s} \sup_\lambda \left| \sum_{j : j^C \gtrsim 2^s} \left( L_{j,s}(\lambda,\beta) \hat{f}(\beta) \right)^{\vee} \right|,
\end{equation}
with $L_{j,s}$ defined in \eqref{Ljs}. 

Our task now is to produce estimates on $\| \eqref{max-s} \|_{\ell^p}$ that sum in $s \geq 1$. After excluding finitely many $s \lesssim_\rho 1$, which we are free to do,\footnote{The issue is estimating the maximal function \eqref{sMax} below in this regime; but, one may simply appeal to Remark \ref{cheaprem}.} we write
\[ \aligned 
&\sum_{j : j^C \gtrsim 2^s} L_{j,s}(\lambda,\beta) \\
& \qquad =
\sum_{\mathcal{R}_s} S(a/q,b/q) \sum_{j : j^C \gtrsim 2^s} H_j(\lambda - a/q, \beta - b/q) \Xi_j(\lambda - a/q) \chi_s(\lambda - a/q) \chi_s(\beta - b/q) \endaligned \]
as a composition of the following two multipliers:
\begin{equation}\label{an}
\sum_{j:j^C \gtrsim 2^s} L^1_{j,s}(\lambda,\beta) 
\end{equation}
and
\begin{equation}\label{ar} L^2_s(\lambda,\beta),
\end{equation}
see \eqref{Lj1} and \eqref{Ls2} above.
The $\ell^p$ norm of
\begin{equation}\label{sMax}
 \sup_\lambda \left| \left( \eqref{an} \hat{f}(\beta) \right)^{\vee} \right| 
\end{equation}
is bounded by a constant multiple of $s 2^{2 s \rho}$ by Theorem \ref{keyest} above. By Proposition \ref{Weylest}, we may bound
\begin{equation}\label{problem}
\sup_{\lambda} \left| \left( L^2_s(\lambda,\beta) \hat{f}(\beta) \right)^{\vee} \right| \end{equation}
on $\ell^p$ by
\[ 2^{-\eta s} \|f\|_{\ell^p}.\]
In particular, we have the following upper bound for the $s$th maximal operator \eqref{max-s}:
\[ \| \eqref{max-s} \|_{\ell^p} \lesssim s \cdot 2^{(2\rho-\eta)s} \|f\|_{\ell^p}, \ \eta = \eta(d,p) > 0;\]
since this estimate sums in $s \geq 1$ for $\rho$ sufficiently small, the proof is complete.

\section{Extensions to $2 - \frac{1}{d^2 + 1} < p \leq 2$}\label{s:comp1}
In this section we sketch how to push our estimates below $p = 2$. This argument is $L^2$-based, and does not rely upon the multiplier theory of \cite{MSZK}, but on multi-frequency estimates when the frequencies form a \emph{cyclic subgroup}, and interpolation.

\subsection{A Multi-Frequency Estimate for Cyclic Subgroups}
Let $s \geq 1$ be an integer, and $Q$ an integer of size $1 \leq Q \leq 2^{s}$. With $k_0 \geq 2^{Cs}$ for some (sufficiently large) $C \gg 1$, for each $0 < \lambda \leq 1$, define $k(\lambda)$ to be the integer, $j$, with
\[ \lambda \cdot j^{-C} 2^{dj} \approx 1 \]
(the implicit constants vary by a multiplicative factor of $2$, see \eqref{impconsts} above). Then, define the function
\begin{equation}\label{FULLOP}
\rho_{\lambda}(t) := \sum_{ k_0 \leq k \leq k(\lambda)} e(\lambda t^d) \psi_k(t);\footnote{\text{Since our estimates will be uniform in $k_0$, we have chosen to suppress the dependence of $\rho_\lambda$ on this parameter.}}
\end{equation}
note that convolution with $\rho_{\lambda}$ are uniformly bounded on $L^p, \ 1 < p < \infty$. 
Finally, for $D_s \gtrsim 2^{Cs}$ for some sufficiently large $C$, define 
\[ \chi_s(\beta) := \chi(D_s \beta),\]
and similarly for $\overline{\chi}_s$.
Consider the following class of maximal functions:
\begin{equation}\label{ArithMax}
P^{Q,s} f := \sup_\lambda \left| \sum_{B \leq Q} e( B/Q x) \rho_\lambda * \left( \chi_s \hat{f}(\cdot + B/Q)\right)^{\vee}(x) \right|.
\end{equation}
By an application of \cite[Corollary 2.1]{MSW} and \cite{SW},\footnote{See the remarks below \eqref{e:4<}} we have the following proposition.
\begin{proposition}\label{ArithMaxEst}
For any $1 < p < \infty$, uniformly in $Q \leq 2^s$, and $s \geq 1$,
\[ \| P^{Q,s} f\|_{\ell^p} \lesssim_p \|f\|_{\ell^p}.\]
\end{proposition}

\subsection{A Multi-Frequency Estimate for General Frequencies: $L^2$ Theory}
We will need the following multi-frequency maximal estimate when the frequencies enjoy no arithmetic structure. In particular, by using the techniques of \S \ref{s:max}, one may establish the following theorem.
\begin{theorem}\label{L2keyest}
With the $T_\lambda$ as defined in \eqref{Tredef}, for any collection of $N$ $\tau$-separated frequencies,
\[ \|T_d f\|_{\ell^2} \lesssim \log^2 N \cdot \|f\|_{\ell^2}.\]
\end{theorem}
We briefly sketch the (minor) changes in the argument of \S \ref{s:max} needed to prove this theorem. First, one subdivides into four cases according to whether $l < -C_d \log N$, $|l| \leq C_d \log N$, and $l > C_d \log N$. For the first case,
one uses Bourgain's estimate for $Mf$, and \cite[Lemma 3.12]{KL} to estimate the multi-frequency maximally truncated the Hilbert transform, rather than Proposition \ref{MSTLp}. The ``critical regime'' becomes straight-forward, as the pertaining square functions may be simply estimated on $\ell^2$ by Plancherel, which reduces the problem to single frequency estimates; in particular, the norm loss in the critical regime, $|l| \lesssim_d \log N$, is
\[ 1 + \min\{ 2^l, 2^{-l}\} \log^2 N,\]
see \cite[Lemma 3.16]{KL}. (Alternatively, one may simply replace the single-frequency multiplier \cite[Lemma 2.9]{KL}, which does not apply when $d \geq 3$, with the square function approach of \S \ref{s:max}; the additional truncations cause no real trouble). The final case is treated using Lemma \ref{l:4}.\footnote{Alternatively, one may simply follow Bourgain's argument, \cite[\S 4]{B3}, substituting the \emph{variational} estimate of \cite{GRY}, see the remarks below Lemma \ref{SingfreqVar}, for the analogous one concerning the Lebesgue averaging operators. Variational estimates will be introduced and discussed in \S \ref{s:pointwise} below.}

\subsection{Completing The Proof, $2 - \frac{1}{d^2 + 1} < p \leq 2$}
Rather than decomposing
\[ L_{j,s} = L_{j,s}^1 L_{s}^2,\]
we replace $L_{j,s}^1$ with
\begin{equation}\label{LBOURG}
L_{j,s}^0(\lambda,\beta) := \sum_{Q < 2^s, \ (A,Q) =1} \sum_{1 \leq B \leq Q} H_j(\lambda - A/Q,\beta-B/Q) \Xi_j(\lambda - A/Q) \chi_s(\lambda - A/Q) \chi_s(\beta - B/Q).
\end{equation}
We need to estimate
\begin{equation}\label{finalest}
\sup_{\lambda} \left| \left( \sum_{j:j^C \gtrsim 2^s} L_{j,s}^0(\lambda,\beta) \hat{f}(\beta) \right)^{\vee} \right|, \end{equation}
or more simply, the majorant of \eqref{finalest},
\begin{equation}\label{finalest0}
\sup_\lambda \left| \left( \sum_{\theta \in \Gamma_s} \sum_{j : j^C \gtrsim 2^s} H_j(\lambda,\beta - \theta) \Xi_j(\lambda) \chi_s(\beta - \theta) \hat{f}(\beta) \right)^{\vee} \right|,
\end{equation}
where
\begin{equation}\label{Gamma}
\Gamma_s := \{ B/Q \text{ reduced}: 1 \leq B \leq Q < 2^s\}
\end{equation}
has $|\Gamma_s| \leq 2^{2s}$.

By Theorem \ref{L2keyest}, the $\ell^2$ norm of \eqref{finalest0} is bounded in norm by 
\[ \lesssim \log^2 |\Gamma_s| \leq \log^2 (2^{2s}) \lesssim s^2.\]
On $\ell^p$ for $p < 2$, we use M\"{o}bius inversion to dominate \eqref{finalest0} by (essentially) $2^s$ multi-frequency maximal functions taken over cyclic subgroups. To this end, we dominate
\[ \aligned
\eqref{finalest0} &\leq \sum_{ Q < 2^s} \sup_\lambda 
\left| \left( \sum_{1 \leq B \leq Q, \ (B,Q) = 1 } \sum_{j : j^C \gtrsim 2^s} H_j(\lambda,\beta - B/Q) \Xi_j(\lambda) \chi_s(\beta - B/Q) \hat{f}(\beta) \right)^{\vee} \right| \\
& \qquad =: \sum_{ Q < 2^s} P_{\mu}^{Q,s} f. \endaligned \]
We will estimate
\[ \| P_\mu^{Q,s} f \|_{\ell^p} \lesssim_\epsilon 2^{\epsilon s} \|f\|_{\ell^p}, 1 < p < \infty;\]
this will allow us to estimate
\begin{equation}\label{end0}
\| \eqref{finalest} \|_{\ell^p} \lesssim_\epsilon 2^{(\epsilon + 2/p -1)s} \|f \|_{\ell^p}
\end{equation}
by interpolation with Theorem \ref{L2keyest}.
To estimate $P_{\mu}^{Q,s}f $, we apply M\"{o}bius inversion:
\begin{equation}\label{Minv}
\sum_{1 \leq a \leq q, \ (a,q) = 1} F(a/q) = \sum_{d|q} \mu(q/d) \sum_{a=1}^d F(a/d)
\end{equation}
to dominate
\[ P_{\mu}^{Q,s}f \leq \sum_{D | Q }
\sup_{\lambda} 
\left| \left( \sum_{B=1}^D  \sum_{j : j^C \gtrsim 2^s} H_j(\lambda,\beta - B/D) \Xi_j(\lambda) \chi_s(\beta - B/D) \hat{f}(\beta) \right)^{\vee} \right|;\]
using the divisor bound,
\[ |\{\text{divisors of $q$} \}| \lesssim_\epsilon q^{\epsilon} \leq 2^{\epsilon s}, \]
we conclude \eqref{end0} by an application of
Proposition \ref{ArithMaxEst}.

But now,
\begin{equation}\label{end1}
\| \sup_\lambda | \left( L_s^2(\lambda,\beta) \hat{f}(\beta) \right)^{\vee}| \|_{\ell^p} \lesssim_\epsilon 2^{(\epsilon - \frac{1}{d^2 p'})s} \|f\|_{\ell^p}, \ 1 <p \leq 2 \end{equation}
by interpolation, so by combining \eqref{end0} and \eqref{end1} we complete the proof of Theorem \ref{MAIN}.

\section{Pointwise Convergence}\label{s:pointwise}
In this section we turn to the \emph{measure-preserving setting}: $(X,\mu,T)$ will denote a $\sigma$-finite measure space, equipped with an invertible measure-preserving transformation, $T$. By transferring an \emph{oscillation inequality} from the integers, a technique introduced by Bourgain in his work on pointwise convergence of ergodic averages along monomial sequences \cite[\S 7]{B1}, we will prove pointwise convergence $\mu$-a.e.
\[ \lim_{\lambda \to 0} \ \sum_{m\neq 0} T^m f \cdot \frac{ e(-\lambda m^d)}{m} =: \lim_{\lambda \to 0} C_\lambda f\]
for $L^2(X)$ functions (in this section we will suppress the dependence on $d$). By transferring our maximal inequality to the measure-preserving setting \cite{C}, we may extend this to all $L^p(X)$ functions, $2 - \frac{2}{d^2+2} < p < \infty$, which will yield Theorem \ref{pointwise}.\footnote{Strictly speaking, for transference purposes, our Carleson operators should have an additional supremum taken over truncations. But, the additional complications arising from this modification are of a formal nature, and in particular the same $L^p$ estimates are obtained, with only minor changes to the argument: the most significant observation is that the analogous continuous (maximally truncated) operator is $L^p$ bounded, see the remarks following Lemma \ref{l:4}.}

Following Bourgain, we will seek a contradiction by assuming that, for any $\epsilon > 0$, for any sequence of intervals
\[ I_i := (2^{-d j_{i+1}}, 2^{-d j_i}], \ 1 \leq i \leq J, \]
we have the estimate
\begin{equation}\label{Bcounter}
\mu \left( \left\{ \sup_{\lambda \in I_i} | C_\lambda f - C_{\lambda_i} f| \gg \epsilon \right\}  \right) \gg \epsilon > 0, 
\end{equation}
where we set $\lambda_i := 2^{-dj_i}$, and $\epsilon > 0$ is an arbitrarily small positive number (independent of $J$).

By our maximal inequality, it suffices to assume that $\|f\|_\infty \leq 1$. One pointwise reduction before we turn to the argument proper. Set $l_i := j_i  + C \log \epsilon$. Then, for $\lambda \in I_i$, we may bound
\[ |C_\lambda f(x) - C_{\lambda_i} f(x)| \leq C\epsilon^C + \left| \sum_{j \geq l_i} \sum_{m} T^m f(x) \psi_j(m) e(-\lambda m^d) - 
\sum_{j \geq l_i} \sum_{m} T^m f(x) \psi_j(m) e(-\lambda_i m^d) \right|. \]
In particular, we will \emph{re-define}
\[ C_\lambda f := \sum_{j \geq l_i} \sum_{m } T^m f \cdot \psi_j(m) e(-\lambda m^d) \]
for $\lambda \in I_i$.
We now present the main result of this section, which will yield the desired contradiction.
\begin{theorem}\label{oscest}
We have the following estimate: there exists an absolute $\kappa > 0$ so that
\[  \sum_{1 \leq i \leq J} \| \max_{\lambda \in I_i} |C_\lambda f - C_{\lambda_i} f| \|_{L^2(X)}^2 \lesssim J^{1-\kappa} \|f\|_{L^2(X)}^2 \]
as $J \to \infty$.
\end{theorem}
\begin{remark}
By transference and the $\ell^2$ boundedness of $\mathcal{C}_df$, we may assume without loss of generality that $i \gg_\epsilon J^{1-\kappa'}$ for some $\kappa' > 0$; in particular, we will assume that $l_i \approx j_i$ throughout.
\end{remark}

By transference, it suffices to prove Theorem \ref{oscest} in the special case of the integer model with the shift:
\[  \sum_{1 \leq i \leq J} \| \max_{\lambda \in I_i} |C_\lambda f - C_{\lambda_i} f| \|_{\ell^2}^2 \lesssim J^{1-\kappa} \|f\|_{\ell^2}^2, \]
where $C_\lambda f$ is now given by
\[ \aligned 
C_\lambda f(x) &:= \sum_{j \geq l_i} \sum_{m} f(x-m) \cdot \psi_j(m) e(-\lambda m^d) \\
&\qquad = \left( \sum_{j \geq l_i} M_j(\lambda,\beta) \hat{f}(\beta) \right)^{\vee}(x). \endaligned \]

A major tool in proving these types of oscillatory estimates are the use of \emph{variational} operators, classically used in probability theory to give quantitative information on rates of convergence, and first used in this context by Bourgain in \cite{B3}.
\begin{definition}
For any sequence of scalars $\{ a_\lambda : \lambda\}$, for any $0 < r < \infty$, we define
\[ \mathcal{V}^r(a_\lambda: \lambda) := \sup \left( \sum_{i} |a_{\lambda_i} - a_{\lambda_{i+1}}|^r \right)^{1/r},\]
where the supremum is taken over all \emph{finite} increasing subsequences; the endpoint 
\[ \mathcal{V}^\infty(a_\lambda: \lambda) := \sup_{\lambda,\mu} |a_\lambda - a_{\mu}| \]
is just defined to by the diameter of the set (which is controlled by $\sup_\lambda |a_\lambda|$).
For a collection of operators $\{ A_\lambda : \lambda \}$ for which
\[ \lambda \mapsto A_{\lambda} f(x) \]
is (almost everywhere) continuous, we define
\[ \mathcal{V}^r(A_\lambda f : \lambda )(x) := \mathcal{V}^r(A_\lambda f(x) : \lambda ).\]
\end{definition}

We remark that if $\V^r(a_\lambda) < \infty$ for some $r < \infty$, we automatically have convergence of the sequence $\{ a_\lambda \}$; in particular, norm estimates for $\V^r(A_\lambda f)$ proves almost everywhere pointwise convergence of the $\{ A_\lambda f\}$. The key inequality that we will use is that, for any collection of intervals $\{ I_1,\dots,I_J\}$, and any operators $\{ A_\lambda \}$ indexed by $\lambda$, we may bound pointwise
\begin{equation}\label{OtoV}
\left( \sum_{1 \leq i \leq J}  \max_{\lambda \in I_i} |A_\lambda f - A_{\lambda_i} f| ^2 \right)^{1/2} \leq J^{1/2-1/r} \cdot \V^r(A_\lambda f), \ r \geq 2;
\end{equation}
in particular, $\ell^2$ estimates on $\V^r(A_\lambda f)$ lead to the types of estimates needed to prove Theorem \ref{oscest}.

The key variational estimate that we will need is the following ``single-frequency" estimate for oscillatory integrals, which is a consequence of \cite[Theorem 1.1]{GRY}, an $r$-variational result for $r,p > 2$, the main result of \cite{SW}, a maximal/``$\infty$-variational'' result for $p$ near $1$, and the interpolation argument of \cite[\S 7]{BK}, see \cite[Figure 1]{GRY}.
\begin{lemma}\label{SingfreqVar}
With $\rho_\lambda$ as in \eqref{FULLOP}, we have the following variational estimate: there exists\footnote{In fact, \emph{every} $2 < r_0 < \infty$ satisfies the below estimate} $2 < r_0 < \infty$ so that
\[ \| \V^{r_0}( \rho_\lambda * f : \lambda) \|_{2} \lesssim \| f \|_2.\]
\end{lemma}
In \cite{GRY}, this result was proven without the spatial truncations, 
\[ k_0 \leq k \leq k(\lambda);\]
but the arguments there are sufficiently robust to extend to this setting with only formal modifications.

With these preliminaries in mind, we turn to the proof of Theorem \ref{oscest}, which we present in the following subsection. This argument follows a similar line to the analogous argument of \cite[\S 6]{B3}.

\subsection{Proof of Theorem \ref{oscest}}
There is no loss of generality in restricting to $i \gtrsim J^{1-\kappa'}$ for some $\kappa'  > 0$, by our $L^2$ theory. We bound
\[ \aligned 
&\sup_{I_i} |C_\lambda f - C_{\lambda_i} f| \\
& \qquad \leq \sup_{I_i} \left| \left( \sum_{j \geq l_i} \big( L_j(\lambda,\beta) - L_{j}(\lambda_i,\beta) \big) \hat{f}(\beta) \right)^{\vee} \right| \\
& \qquad \qquad + 2 \sum_{j \geq l_i} \sup_\lambda \left| \left( \mathcal{E}_j(\lambda,\beta) \hat{f}(\beta) \right)^{\vee} \right|. \endaligned\]
By Proposition \ref{t:smooth} and Lemma \ref{sobemb}, the $\ell^2$ norm of the second term on the right is 
\[ \lesssim l_i^{-C} \lesssim j_i^{-C} \lesssim i^{-C} \lesssim J^{-\kappa} \]
for some absolute $\kappa > 0$, so we may discard the sum in $J^{1-\kappa'} \leq i \leq J$ of the error terms.
We now majorize
\[ \aligned 
& \sup_{I_i} \left| \left( \sum_{j \geq l_i} \big( L_j(\lambda,\beta) - L_{j}(\lambda_i,\beta) \big) \hat{f}(\beta) \right)^{\vee} \right| \\
& \qquad \leq \sup_{I_i} \left| \left( \sum_{s : 2^s \lesssim l_i^C} \sum_{j \geq l_i} \big( L_{j,s}(\lambda,\beta) - L_{j,s}(\lambda_i,\beta) \big) \hat{f}(\beta) \right)^{\vee} \right| \\
& \qquad \qquad +
\sup_{I_i} \left| \left( \sum_{s : 2^s \gg l_i^C} \sum_{j^C \gtrsim 2^s} \big( L_{j,s}(\lambda,\beta) - L_{j,s}(\lambda_i,\beta) \big) \hat{f}(\beta) \right)^{\vee} \right|.
\endaligned\]
If we majorize this final term by a constant multiple of
\[ \sum_{s : 2^s \gg l_i^C} \sup_\lambda \left| \left( \sum_{j^C \gtrsim 2^s} L_{j,s}(\lambda,\beta) \hat{f}(\beta) \right)^{\vee} \right|,\]
which has $\ell^2$ norm
\[ \lesssim \sum_{s : 2^s \gg l_i^C} s^2 2^{(\epsilon - \frac{1}{2d^2}) s} \lesssim
l_i^{- \kappa'} \lesssim j_i^{-\kappa'} \lesssim J^{-\kappa},\]
we see that it suffices to prove
\begin{equation}\label{oscgoal}
\sum_{J^{1-\kappa'}}^J \| \sup_{I_i} \left| \left( \sum_{s : 2^s \lesssim l_i^C} \sum_{j \geq l_i} \big( L_{j,s}(\lambda,\beta) - L_{j,s}(\lambda_i,\beta) \big) \hat{f}(\beta) \right)^{\vee} \right\|_{\ell^2}^2 \lesssim J^{1-\kappa}  \|f\|_{\ell^2}^2,
\end{equation}
In fact, by the triangle inequality, Proposition \ref{Weylest}, and Cauchy-Schwartz, it suffices to prove
\begin{equation}\label{oscgoal}
\sum_{s \geq 1} 2^{(\epsilon - 1/d^2)s} \ \times \ \sum_{J^{1-\kappa'} \leq i \leq J : l_i^C \gtrsim 2^s} \| \sup_{I_i} \left| \left( \sum_{j \geq l_i} \big( L^0_{j,s}(\lambda,\beta) - L^0_{j,s}(\lambda_i,\beta) \big) \hat{f}(\beta) \right)^{\vee} \right\|_{\ell^2}^2 \lesssim J^{1- \kappa} \|f\|^2_{\ell^2}.
\end{equation}
By our $\ell^2$ theory, we may assume that $s \leq \frac{\kappa'}{r_0} \log J$, since the contribution above this cut-off is bounded by a constant multiple of
\[ \left( \frac{\kappa'}{r_0} \log J \right) ^4 2^{(\epsilon -1/d^2) \cdot \left( \frac{\kappa'}{r_0} \log J \right)} \times J \times \|f\|_{\ell^2}^2 \lesssim J^{1 - \kappa''} \times \|f\|_{\ell^2}^2. \]

We now need to estimate, for $s \leq \frac{\kappa'}{r_0} \log J$,
\begin{equation}\label{oscsquare}
\left( \sum_{i = J^{1-\kappa'}}^J |T_{s,i} f|^2 \right)^{1/2}
\end{equation}
in $\ell^2$, where  
\[ T_{s,i} f := \sup_{\lambda \in I_i} \left| \left( \sum_{j \geq l_i} \big( L_{j,s}^0(\lambda,\beta) - L_{j,s}^0(\lambda_i,\beta) \big) \hat{f}(\beta) \right)^{\vee} \right| \times \mathbf{1}_{i : l_i^C \gtrsim 2^s}. \]
We are now free to use the fact that we are only dealing with an acceptably small number of frequencies to dominate
\[ \aligned 
\eqref{oscsquare} &\leq 
\sum_{n=1}^{N_s} \left( \sum_{i = J^{1-\kappa'}}^J \sup_{\lambda \in I_i} | \big( {\rho}_\lambda - {\rho}_{\lambda_i} \big) * \left( \chi_s \hat{f}(\cdot + \theta_n) \right)^{\vee} |^2 \right)^{1/2} \\
& \qquad \leq \sum_{n=1}^{N_s} J^{1/2 - 1/r_0} \cdot \V^{r_0} \left( {\rho}_\lambda * \left( \chi_s \hat{f}(\cdot + \theta_n) \right)^{\vee} : \lambda \right),
\endaligned
\]
where $r_0$ is as in Lemma \ref{SingfreqVar}, and $N_s = |\Gamma_s| \leq 2^{2s}$. By Lemma \ref{SingfreqVar} and Lemma \ref{trans}, we may estimate
\[ \| \V^{r_0} \left( {\rho}_\lambda * \left( \chi_s \hat{f}(\cdot + \theta_n) \right)^{\vee} : \lambda \right) \|_{\ell^2} \lesssim \|f\|_{\ell^2},\]
so we may bound
\[ \aligned 
\| \eqref{oscsquare} \|_{\ell^2} &\lesssim \sum_{s \leq \frac{\kappa'}{r_0} \log J} 2^{2s} \cdot J^{1/2 - 1/r_0} \|f\|_{\ell^2} + J^{1/2 - \kappa'} \|f\|_{\ell^2} \\
& \qquad \lesssim J^{1/2 - \kappa} \|f\|_{\ell^2}, \endaligned\]
which completes the proof of convergence.

\section{Appendix: The Proof of Lemma \ref{decomp}}\label{appendix}
The (technical) proof of Lemma \ref{decomp}, reproduced below, will follow from stationary phase considerations. 
\begin{lemma}
For any (large) $N$, one may decompose $G_\lambda = A_\lambda + \sum'_{\pm} B^{\pm}_{\lambda}$, which satisfy the following estimates, independent of $\lambda$:
\[ |A_\lambda*f| \lesssim_N 2^{-lN} M_{HL} \left( \overline{\zeta}(2^{k-l}\cdot) \hat{f} \right)^{\vee} \]
pointwise,
and
\[ \widehat{B^{\pm}_{\lambda}}(\xi) = 2^{-l/2} \cdot e( \pm c_d \lambda^{-1/(d-1)} \xi^{d/(d-1)} ) \cdot m(\xi,\lambda) \cdot \zeta(2^{k-l}\xi),\]
for some $|c_d| \approx_d 1$; in the case where $d$ is odd, we replace $\zeta$ with $\zeta \cdot \mathbf{1}_{\xi < 0}$ throughout (which satisfies all the same differential inequalities as does $\zeta$ itself).
Here
\[ \sup_\lambda |\partial_\xi^j m(\xi,\lambda)| \lesssim_j |\xi|^{-j}, \ j \geq 0, \ \xi \neq 0.\]
In particular, we may decompose
\[ \widehat{B^{\pm}_{\lambda}}(\xi) = 
\widehat{O^{\pm}_\lambda(\xi)} \widehat{ M_\lambda (\xi) },\]
where
\[ \widehat{O^{\pm}_\lambda(\xi)} := 2^{-l/2} \cdot e( \pm c_d \lambda^{-1/(d-1)} \xi^{d/(d-1)} ) \cdot \zeta(2^{k-l}\xi),\]
and
\[ \widehat{ M_\lambda (\xi) } := m(\xi,\lambda) \cdot \overline{\zeta}(2^{k-l}\xi) \]
satisfies
\[ |M_\lambda(x) | \lesssim_N 2^{l-k} (1 + |2^{l-k} x|)^{-N}.\]
Here $\sum'_{\pm}$ means that the ``minus'' term appears only when $d$ is odd.
\end{lemma}

The key point is that the phase 
\[ \eqref{phase} = \varphi^k(t,\xi) =
-2^{-l} \left( \lambda 2^{kd} t^{d} + \xi 2^k t\right) \]
has a \emph{critical point} (possibly two) in $t$ at $|t(\xi)| = |t(\lambda,\xi)| \approx 1$, where $t(\xi)$ is defined via the relationship
\begin{equation}\label{imp} 
d \lambda 2^{k(d-1)} t(\xi)^{d-1} = -\xi.
\end{equation}
In the case where $d$ is even, $t(\xi)$ is uniquely defined,
\[ t(\xi) = - (d\lambda)^{-1/(d-1)} \times 2^{-k} \times \text{sgn}(\xi) |\xi|^{1/(d-1)},\]
otherwise $t(\xi)$ is given by
\[ (d\lambda)^{-1/(d-1)} \times 2^{-k} \times \pm |\xi|^{1/(d-1)},\]
when $\xi < 0$; otherwise the phase has no critical points.

There are only $\leq 2$ solutions to \eqref{imp}, we will only work with the maximal one, as the other can be treated similarly. Note that $t(\xi)$ is an analytic function of $\xi$ for $|\xi| \approx 2^{l-k}$ (and thus $|t(\xi)| \approx 1$).

Differentiating \eqref{imp} with respect to $\xi$ leads to the identity
\[ (d-1) \xi \cdot \partial_\xi t(\xi) = - t(\xi),\]
or
\begin{equation}\label{impderiv}
\partial_\xi t(\xi) = - \frac{t(\xi)}{(d-1) \xi}.
\end{equation}

If we Taylor expand $\varphi^k(t,\xi)$ about $t(\xi)$, we have
\[ \aligned \varphi^k(t,\xi) &= \varphi^k(t(\xi),\xi) + \sum_{j=2}^d \frac{ \partial_t^j \varphi^k(t(\xi),\xi) }{j!} (t-t(\xi))^j \\
& \qquad = -\frac{d+1}{d} 2^{k-l} t(\xi) \cdot \xi + \sum_{j=2}^d \frac{ \partial_t^j \varphi^k(t(\xi),\xi) }{j!} (t-t(\xi))^j \\
& \qquad \qquad = -(d+1)\lambda 2^{kd-l} t(\xi)^d + \sum_{j=2}^d \frac{ \partial_t^j \varphi^k(t(\xi),\xi) }{j!} (t-t(\xi))^j. 
\endaligned
\]
Inserting this into the bracketed expression in \eqref{invF}, we decompose into the inverse Fourier transform of two terms, as in Proposition \ref{decomp}:
\begin{equation}\label{global}
\widehat{A_\lambda}(\xi) := e( - (d+1)\lambda 2^{kd} t(\xi)^d  ) \int e(2^l \sum_{j=2}^d \frac{ \partial_t^j \varphi^k(t(\xi),\xi) }{j!}s^j) \psi(s+t(\xi)) (1-\Xi_0(s)) \ ds \cdot \zeta(2^{k-l} \xi)
\end{equation}
and
\begin{equation}\label{local}
\widehat{B_\lambda^+}(\xi) := e( - (d+1)\lambda 2^{kd} t(\xi)^d  ) \int e(2^l \sum_{j=2}^d \frac{ \partial_t^j \varphi^k(t(\xi),\xi) }{j!}s^j) \psi(s+t(\xi)) \Xi_0(s) \ ds \cdot \zeta(2^{k-l} \xi)
\end{equation}
where $\Xi_0$ is a smooth approximation to the indicator function of a tiny ball near the origin. We will assume that the support of $\Xi_0$ is sufficiently small -- independent of $k,l$. We have the following stationary phase lemma.
\begin{lemma}\label{glo}
For any $\lambda \approx 2^{l-dk}$,
\[ |A_\lambda(x)| \lesssim_N \begin{cases} 2^{-lN -k} &\mbox{if } |x| \lesssim 2^k  \\ 
2^{l-k} (2^{l-k} |x|)^{-N} & \mbox{if } |x| \gg 2^k. \end{cases}\]
\end{lemma}
\begin{proof}
Since we have excised the critical point in the support of the integral in \eqref{global}, we have the pointwise estimate,
\[ |\widehat{A_\lambda}| \lesssim_N 2^{-lN} \mathbf{1}_{|\xi| \approx 2^{l-k}},\]
which yields,
\[ |A_\lambda(x)| \lesssim_N 2^{-lN -k};\]
we will use this bound when $|x| \lesssim 2^k$. In the complementary regime, by change of variables, we see that there exist a collection of non-zero constants,
\[ \{ c_\lambda, c'_\lambda, c_{\lambda,1},\dots,c_{\lambda,d}\}, \]
all about $1$ in magnitude, uniformly in $\lambda$ (these constants vary linearly in $\lambda$, and we have normalized our phase appropriately)
so that we may write
\[ A_\lambda(x) = c_\lambda 2^{l-k} \int \Omega(s,t) e(2^l \cdot \Phi(x,s,t)) \ dsdt,\]
where
\[ \Omega(s,t) = (1-\Xi_0)(s) \psi(s+t) \zeta(c'_\lambda t^{d-1} ) t^{d-2} \]
is a nice bump function, and the phase $\Phi(x,s,t)$ is given by
\[ \Phi(x,s,t) := c_{\lambda,1}(2^{-k}x - \frac{d+1}{d} t)t^{d-1} + \sum_{j=2}^d c_{\lambda,j} t^j s^j.\]
But, when $|x| \gg 2^k$, $|\nabla \Phi(x,s,t)| \approx |x 2^{-k}|$ uniformly on the support of $\Omega(s,t)$, so the result follows by the principle of non-stationary phase.
\end{proof}

We now turn to an analysis of $B_\lambda^+$. As per Proposition \ref{decomp}, we will view $B_\lambda^+$ as a product
\begin{equation}\label{osc}
2^{-l/2} \times e(-(d+1) \lambda 2^{kd} t(\xi)^d) \zeta(2^{k-l}\xi)
\end{equation}
and
\begin{equation}\label{mult}
2^{l/2} \times \int e(2^l \sum_{j=2}^d \frac{ \partial_t^j \varphi^k(t(\xi),\xi) }{j!}s^j) \psi(s+t(\xi)) \Xi_0(s) \ ds \cdot \overline{\zeta}(2^{k-l} \xi),
\end{equation}
where we replace $\overline{\zeta}$ by $\overline{\zeta} \mathbf{1}_{\xi < 0}$ if $d$ is odd.

But, \eqref{osc} is precisely
$\widehat{O^+_\lambda}(\xi),$
upon replacing
\[ -(d+1) \lambda 2^{kd} t(\xi)^d = - \frac{d+1}{d} \cdot d^{-1/(d-1)} \times \lambda^{-1/(d-1)} \xi^{d/(d-1)} =: c_d \lambda^{-1/(d-1)} \xi^{d/(d-1)},\]
where $\xi^{d/(d-1)}$ is defined in \eqref{Fracpower}.

We turn to \eqref{mult}, which we need to prove has an appropriately decaying inverse-Fourier transform (adapted to spatial scales $2^{k-l}$).
%\begin{lemma} With the inverse Fourier transform taken only in $x$, we may bound
%\[ |O^{\vee}(x,\lambda)| \lesssim_N 2^{-lN} \cdot 2^{-k} (1 + |2^{-k}x|)^{-N} + 2^{l/2} \cdot 2^{-k} \mathbf{1}_{|x| \approx 2^k}.
%\]
%\end{lemma}
%\begin{proof}
%It suffices to estimate
%\[ 2^{l-k} \int e(2^{l-k} \xi x \pm c\lambda^{-1/(d-1)} 2^{(l-k)d/(d-1)} |\xi|^{d/(d-1)}) \zeta'(\xi) \ dx, \]
%and the result follows from the principle of non-stationary phase unless $|x| \approx 2^k$, in which case the result follows from the principle of stationary phase. 
%\end{proof}
In particular, we will prove that \eqref{mult} is a dyadic piece of a (regular enough) Calder\'{o}n-Zygmund kernel, and hence its inverse Fourier transform is bounded by an averaging operator at the natural scale:
\begin{lemma}\label{MULTEST}
For any $j \geq 0$, we may estimate
\[ | \partial_\xi^j \eqref{mult} | \lesssim_j |\xi|^{-j} \mathbf{1}_{|\xi| \approx 2^{l-k}},\]
and thus
\[ |\eqref{mult}^{\vee}(x)| \lesssim_N 2^{l-k} (1 + 2^{l-k}|x|)^{-N}.\]
\end{lemma}

Once we prove this lemma, we may conclude Proposition \ref{decomp}.

Here is the strategy:

Roughly speaking, we have
\[ \eqref{mult} = 2^{l/2} \int e(2^l \cdot \Phi(\xi,s) ) \psi(s+t(\xi)) \Xi_0(s) \ ds \cdot \zeta(2^{k-l} \xi),\]
for a phase function $\Phi(\xi,s)$ which has a non-degenerate critical point at $s =0$. The plan is to make a change of variables
\[ G := \Phi(\xi,s)^{1/2},\]
so that we may express the integral as
\[ \int e(2^l G^2 ) \psi'(G,\xi)  \ dG = c \cdot 2^{-l/2} \int e(2^{-l} \eta^2) \cdot \F_G(\psi'(\cdot,\xi))(\eta) \ d\eta,\]
for some other bump $\psi'$, which depends on $t(\xi)$, and hence $\xi$. But, we will need to make sure that our resulting functions are rather \emph{smooth} in $\xi$, so we need to pay special attention to our change of variables function $G$.

To do so, we consider the \emph{analytic} function of $(s,\xi)$
\[ \aligned 
& \sum_{j=2}^d \frac{ \partial_t^j \varphi^k(t(\xi),\xi) }{j!} s^j \\
& \qquad =
\frac{\partial_t^2 \varphi^k(t(\xi),\xi)}{2} s^2 \cdot \left(1 + \sum_{j=3}^d \frac{2}{j!} \cdot \frac{\partial_t^j \varphi^k(t(\xi),\xi) }{\partial_t^2 \varphi^k(t(\xi),\xi)} s^{j-2} \right) \\
& \qquad \qquad \frac{\lambda 2^{dk}(d)(d-1) t(\xi)^{d-2}}{2^{l+1}} s^2 \cdot \left( 1 + \sum_{j=3}^d \frac{\binom{d}{j}}{d(d-1)} t(\xi)^{2-j} s^{j-2} \right) \\
& \qquad \qquad \qquad =: G(s,\xi)^2. \endaligned\]
In particular:
\begin{equation}\label{G}
G(s,\xi) := 
\sqrt{ \frac{\lambda 2^{dk}(d)(d-1)}{2^{l+1}} } t(\xi)^{\frac{d-2}{2}} s \left( 1 + \sum_{j=3}^d \frac{\binom{d}{j}}{d(d-1)} t(\xi)^{2-j} s^{j-2} \right)^{1/2}
\end{equation}
is an analytic function of $(s,\xi)$ that is a local analytic isomorphism. Possibly after decreasing the support of $\Xi_0$, we may assume that $G(s,\xi)$ is in fact an analytic isomorphism for all $|\xi| \approx 2^{l-k}$ (which corresponds to $|t(\xi)| \approx 1$) on the support of $\Xi_0$.

For every $\xi$, let $H(u,\xi)$ be the analytic inverse of $G(s,\xi)$: 
\[ H(G(s,\xi),\xi) = s, \ G(H(u,\xi),\xi) = u, \ \text{ for all } \xi.\]
We claim that $H$ is in fact analytic in $\xi$ as well. Indeed, consider the map
\[ \mathbb{C}^2 \ni (s,\xi) \mapsto (G(s,\xi),\xi) \in \mathbb{C}^2.\]
The Jacobian of this change of variables is $\partial_s G(s,\xi)$ which is bounded away from zero on its domain, since we have assumed that $s$ is sufficiently small on the support of $\Xi_0$. Consequently, $G \otimes \mathbf{1}$ has an analytic inverse, namely $H \otimes \mathbf{1}$, which exhibits $H$ as analytic. The key point for us is that $H(u,\xi)$ and all its derivatives converge on the same domain. We will need the following lemma.
\begin{lemma}
For any $L \geq 0$,
\[ \partial_\xi^j \partial_u^L H(u,\xi) = \frac{1}{\xi^j} H^j_L(u,\xi), \ j \geq 1\]
where $H^j_L(u,\xi)$ is an analytic function with the same radius of convergence of $H(u,\xi)$, whose coefficients are uniquely determined by those of $H(u,\xi)$.
\end{lemma}
\begin{proof}
To see this, we begin by expanding $G(s,\xi)$ in power series expansion about $s$, so that we have 
\[ G(s,\xi) = \sum_{j=1}^\infty P_j(t(\xi)) s^j \]
for some polynomials $P_j$. But, by the chain rule,
\[ \partial_\xi P_j(t(\xi)) = \tilde{P_j}(t(\xi))/\xi \]
for some other polynomial $\tilde{P_j}$ of the same degree.
If we expand $H(u,\xi)$ in a power series in $u$,
\[ H(u,\xi) = \sum_{j = 1}^\infty b_j(\xi) u^j,\]
then by inspection of coefficients, we see that $b_j(\xi)$
is a polynomial in 
\[ P_1(t(\xi)), \dots, P_j(t(\xi)),\] i.e.\ is itself a polynomial in $t(\xi)$. Consequently,
\[ \partial_\xi b_j(\xi) = \frac{ \tilde{b_j}(\xi) }{\xi},\]
where
\[ \partial_\xi H(u,\xi) = \frac{1}{\xi} \sum_{n=1}^\infty \tilde{b_j}(\xi) u^j \]
has the same radius of convergence as $H(u,\xi)$, and $\tilde{b_j}(\xi)$ are still polynomials in $t(\xi)$, so one may differentiate again and conclude the result for $j \geq 2$, $L =0$ by induction. But, by commutativity of mixed partials, the result for higher $L \geq 1$ follows.
\end{proof}

With these remarks in mind, we return to \eqref{mult},
\[
2^{l/2} \int e(2^l \sum_{j=2}^d \frac{ \partial_t^j \varphi^k(t(\xi),\xi) }{j!}s^j) \psi(s+t(\xi)) \Xi_0(s) \ ds \cdot \zeta(2^{k-l} \xi).
\]
\begin{proof}[Proof of Lemma \ref{MULTEST}]
We begin by making the substitution given by \eqref{G}, to express
\[ \eqref{mult} = 2^{l/2} \int e(2^l G^2) \psi(H(G)+t(\xi)) \Xi_0( H(G) ) H'(G) \ dG \cdot \zeta(2^{k-l} \xi),\]
where we have suppressed the dependence of $H$ on $\xi$, and we let $H'(G) = (\partial_u H)(G,\xi)$ denote the partial derivative in the first variable.
If we apply the Fourier transform (in $G$), we may express the foregoing as 
\begin{align}
\int e(2^{-l \eta^2}) \mathcal{F}_G \Big( \psi(H(\cdot)+t(\xi)) \Xi_0( H(\cdot) ) H'(\cdot) \Big)(\eta) \ d\eta \cdot \zeta(2^{k-l} \xi).
\end{align}
Since $\psi(H(G)+t(\xi)) \Xi_0( H(G) ) H'(G)$ is a Schwartz function (of two variables), we may differentiate $\eqref{mult}$ under the integral to deduce
\[ |\partial_\xi^j \eqref{mult} | \lesssim_j |\xi|^{-j} \mathbf{1}_{|\xi| \approx 2^{l-k}}, \ j \geq 0.\]
as desired.
\end{proof}

\typeout{get arXiv to do 4 passes: Label(s) may have changed. Rerun}

\end{document}